\documentclass[letterpaper]{amsart}
\usepackage[utf8]{inputenc}
\usepackage[T1]{fontenc}

\usepackage{times}

\usepackage[pagebackref,colorlinks=true,pdfpagemode=UseNone,urlcolor=blue,
linkcolor=blue,citecolor=blue]{hyperref}

\usepackage{amsmath,amsfonts,amssymb,amsthm}
\usepackage{amsthm, amssymb, amsmath, amsfonts, mathrsfs}

\usepackage{mathrsfs}
\usepackage{MnSymbol}
\usepackage{scalerel} 

\usepackage{color}
\usepackage{accents}

\usepackage{bbm}

\usepackage[normalem]{ulem}

\usepackage[notcite,notref,color]{showkeys}
\definecolor{labelkey}{gray}{.8}
\definecolor{refkey}{gray}{.8}

\definecolor{darkred}{rgb}{0.9,0.1,0.1}
\definecolor{darkgreen}{rgb}{0,0.5,0}

\newcommand{\tnorm}[1]{{\left\vert\kern-0.25ex\left\vert\kern-0.25ex\left\vert #1 
    \right\vert\kern-0.25ex\right\vert\kern-0.25ex\right\vert}}

\newtheorem{theorem}{Theorem}[section]
\newtheorem{lemma}[theorem]{Lemma}

\newtheorem{proposition}[theorem]{Proposition}

\theoremstyle{remark}
\newtheorem{remark}[theorem]{Remark}

\renewenvironment{proof}[1][Proof]{ {\itshape \noindent {#1.}} }{$\Box$
\medskip}

\numberwithin{equation}{section}
\newcommand{\R}{\mathbb{R}}

\newcommand{\Z}{\mathbb{Z}}

\newcommand{\Pb}{\mathbb{P}}
\newcommand{\PP}{\mathbf{P}}
\newcommand{\E}{\mathbb{E}}

\newcommand{\F}{\mathcal{F}}

\newcommand{\B}{\mathcal{B}}

\newcommand{\G}{\mathcal{G}}

\newcommand{\U}{\mathcal{U}}

\newcommand{\eps}{\varepsilon}

\newcommand{\EE}{\mathbf{E}}

\newcommand{\bQ}{\mathbbm{Q}}

\def\blue{\textcolor{black}}

\newcommand{\bT}{\mathbb{T}}
\newcommand{\cZ}{\mathcal{Z}}

\newcommand{\cX}{\mathcal{X}}

\newcommand{\rhof}{\rho_{\mathrm{f}}}
\newcommand{\rhob}{\rho_{\mathrm{b}}}

\newcommand{\scU}{\mathscr{U}}

\newcommand{\ic}{\mathrm{ic}}

\begin{document}
\title{Jointly stationary solutions of periodic Burgers flow}

\author{Alexander Dunlap\and Yu Gu}

\address[Alexander Dunlap]{Department of Mathematics, Duke University, Durham, NC 27708, USA.}
\email{dunlap@math.duke.edu}

\address[Yu Gu]{Department of Mathematics, University of Maryland, College Park, MD 20742, USA. }
\email{yugull05@gmail.com}


\begin{abstract} 
For the one dimensional Burgers equation with a random and periodic forcing, it is well-known that there exists a family of invariant measures, each corresponding  to a different average velocity. In this paper, we consider the coupled invariant measures and study how they change as the  velocity parameter varies. We show that the derivative of the invariant measure with respect to the velocity parameter exists, and it can be interpreted as the steady state of a diffusion advected  by the Burgers flow. 
\medskip

\end{abstract}
\maketitle


\section{Introduction}

\subsection{Main result}

Consider the solutions to the stochastic Burgers equation with a family of constant initial data: for any $\theta\in\R, T\in\R$, let $\{u_{\theta}(t,x;-T)\}_{t\geq -T,x\in\R}$ be the solution to
\begin{equation}\label{e.burgers}
\begin{aligned}
\partial_t u_{\theta}(t,x;-T)&=\frac12\Delta u_{\theta}(t,x;-T)+\frac12\nabla ( u_{\theta}^2)(t,x;-T)+\nabla \xi(t,x), && t>-T, x\in\bT,\\
u_{\theta}(-T,x;-T)&=\theta.
\end{aligned}
\end{equation}
Here $\bT=\R/\Z$ is the unit one-dimensional torus (circle) and $\xi$ is a Gaussian noise that is white in time, smooth and $1$-periodic in space, and built on a probability space $(\Omega,\F,\PP)$. The covariance function of $\xi$ is given by 
\[
    \EE [\xi(t,x)\xi(s,y)]=\delta(t-s)R(x-y),
\]
where $R(\cdot)$ is a smooth and $1$-periodic function. The above equation \eqref{e.burgers} preserves the \blue{spatial integral of $u$:} for any $t>-T$, we have 
\begin{equation}\label{e.l1theta}
\int_{\bT} u_\theta(t,x;-T)dx=\theta.
\end{equation}

It is well-known from the early work of Sinai \cite{sinai} that the one-force-one-solution principle holds, namely, as $T\to\infty$, $u_\theta(t,x;-T)$ converges 
to some limit, denoted by $\mathscr{U}_\theta(t,x)$, which is the global solution to the stochastic Burgers equation with mean $\theta$. 

In this paper, we are interested in the coupled stationary solutions to the Burgers equation $\{\mathscr{U}_\theta(t,x):t\in\R,x\in\bT\}_{\theta\in\R}$, and how they behave as the parameter $\theta$ varies. One of our main results is on the differentiability of $\scU_{\theta}(t,x)$ in the parameter $\theta$:

\begin{theorem}\label{t.mainth}
(i) For any $\theta\in\R$, there exists a unique positive global solution $\tilde{g}_\theta$ to 
\begin{equation}\label{e.fk1}
    \partial_t \tilde{g}_\theta(t,x)=\frac12\Delta \tilde{g}_\theta(t,x)+\nabla(\scU_\theta\tilde{g}_\theta)(t,x), \quad\quad  t\in\R, x\in\bT,
\end{equation}
that is spacetime stationary and satisfies $\EE\, \tilde{g}_\theta(0,0)=1$, $\tilde{g}_\theta(0,0)\in L^p(\Omega)$ for any $p\in [1,\infty)$, and, with probability $1$, $\int_{\bT} \tilde{g}_\theta(t,x)dx=1$ for each $t\in\R$.

(ii) For any $\lambda\in\R$ and $t\in\R,x\in\bT$, the following identity holds almost surely:
\begin{equation}\label{e.mainidentity}
\scU_{\lambda}(t,x)-\scU_0(t,x)=\int_0^\lambda \tilde{g}_\theta(t,x)d\theta.
\end{equation}
\end{theorem}

The PDE \eqref{e.fk1} can be immediately obtained, at least formally, by differentiating the SPDE \eqref{e.burgers} with respect to $\theta$. (See Proposition~\ref{p.fpg} below.) On the other hand, the reader will also recognize \eqref{e.fk1} as the Fokker--Planck equation for the SDE
\begin{equation}\label{e.defYt}
dY_t=-\scU_\theta(t,Y_t)dt+dB_t, \quad\quad Y_0=0,
\end{equation}
where $\{B_t\}_{t\geq0}$ is a standard Brownian motion, independent of $\mathscr{U}_\theta$. That is, the quenched density of $Y_t$ (i.e.\ the density with $\scU_\theta$ considered as fixed) satisfies \eqref{e.fk1}.
Thus, there is a natural interpretation of $\tilde{g}_\theta$, as the Radon-Nikodym derivative of the ``environment seen from the particle'' for a diffusion in the Burgers drift ``$-\scU_\theta$''.

To state this result, we introduce a few notations. Fix $\theta\in\R$. We view $\scU_\theta$ as a spacetime random environment, and for the convenience of notations, we denote it by $\omega$, i.e., each instance of $\omega$ is an element in $C(\R\times\bT)$, and $\omega(t,x)=\scU_\theta(t,x)$. Since $\scU_\theta$ is spacetime stationary, there exists a group of measure-preserving transformations $\{\tau_{(t,x)}\}_{t\in\R,x\in\bT}$ on $\Omega$ such that 
\[
\tau_{(t,x)}\omega(\cdot,\cdot)=\omega(t+\cdot,x+\cdot).
\]

Now we consider the diffusion $Y_t$ satisfying \eqref{e.defYt}.
The environment seen from the particle is 
\[
\omega_t(\cdot,\cdot):=\tau_{(t,Y_t)}\omega(\cdot,\cdot)=\omega(t+\cdot,Y_t+\cdot),
\]
which is a Markov process taking values in $C(\R\times\bT)$. Since $Y_0=0$, we have $\omega_0=\omega$. Note that the processes $\omega_t$ and $Y_t$ both depend on the parameter $\theta$, and we have kept the dependence implicit to simplify the notation. 

Let $\mathscr{P}$ be the law of $\mathscr{U}_\theta$, considered as a probability measure on $C(\R\times\bT)$. Since $\tilde{g}_\theta(0,0)$ is a positive random variable with mean $1$, we consider it as a Radon-Nikodym derivative and define the probability measure $\mathscr{Q}$ on $C(\R\times \bT)$ satisfying
\[
\frac{d\mathscr{Q}}{d\mathscr{P}}=\tilde{g}_\theta(0,0).
\]
Here is the main result on $\omega_t$:
\begin{theorem}\label{t.envi}
$\mathscr{Q}$ is an invariant measure of the Markov process $\{\omega_t\}_{t\geq0}$, and, as $t\to\infty$,  $\omega_t$ converges in distribution to $\mathscr{Q}$.
\end{theorem}

\subsubsection*{The spacetime white noise}

It is natural to ask about the singular case when $\xi$ is a spacetime white noise, with $R(\cdot)=\delta(\cdot)$. With the singular noise, the solution to the Burgers equation is distribution-valued, and the Fokker-Planck equation \eqref{e.fk1} is only formal. Instead of $u_\theta$, we can consider its spatial integral, \blue{which is the increments of the solution to the KPZ equation, is function-valued and related to the so-called \emph{Busemann function} in the first/last passage percolation and the model of directed polymers, see e.g.\ \cite{DH14,GRT17,GRSS23,CH05,CH08,JR20}. Our definition below can be viewed as a special case of the Busemann function for the directed polymer model associated with \eqref{e.burgers}.}

For any $t,x,T,\theta$, define 
\begin{equation}\label{e.busemann}
B_{\theta}(t,x;-T):=\int_0^x u_\theta(t,y;-T)dy,
\end{equation}
and its limit as $T\to\infty$ by $\mathscr{B}_\theta(t,x)$.
Here is our main result in this case:
\begin{theorem}\label{t.white}
\blue{In the case of $\xi$ being a spacetime white noise,} there exists a process $\{\psi_\theta(t,x)\}_{\theta,t,x}$ such that 
 for any $t,\lambda\in\R,x\in\bT$, we have 
\[
\mathscr{B}_\lambda(t,x )-\mathscr{B}_0(t,x )=\int_0^\lambda [x+\tilde{\psi}_\theta(t,x)]d\theta, 
\]
which holds almost surely. 
For fixed $t$ and $x$, the process $\{\tilde{\psi}_\theta(t,x)\}_{\theta\in\R}$ is stationary in $\theta$; for fixed $\theta$, its law is given by \begin{equation}\label{e.lawBB}
\{x+\tilde{\psi}_\theta(t,x)\}_{x\in\bT}\stackrel{\text{law}}{=}\left\{\frac{\int_0^x e^{\B_1(y)+\B_2(y)}dy}{\int_0^1 e^{\B_1(y)+\B_2(y)}dy}\right\}_{x\in\bT},
\end{equation}
where $\B_1,\B_2$ are independent standard Brownian bridges connecting $(0,0)$ and $(1,0)$.
\end{theorem}

On a heuristic level, Theorem~\ref{t.white} can be viewed as a special case of Theorem~\ref{t.mainth}, and the density function $\frac{e^{\B_1(\cdot)+\B_2(\cdot)}}{\int_0^1 e^{\B_1(y)+\B_2(y)}dy}$ appearing on the r.h.s.\ of \eqref{e.lawBB} plays the role of the $\tilde{g}_\theta(t,\cdot)$ in \eqref{e.mainidentity}, see more discussion in Remark~\ref{r.density} below.

\subsection{Context}
As a simplified model of hydrodynamic turbulence, the stochastic Burgers equation  has been a subject of extensive study since the early 1990s. Its importance, however, extends beyond hydrodynamics; this equation also holds a critical position in non-equilibrium statistical physics, especially in the study of random growth models. It is the derivative of the Kardar-Parisi-Zhang (KPZ) equation, which is one of the default models of interface growth subject to random perturbations. The Burgers and KPZ equations have been studied intensively in recent years, and they appear as the scaling limits of a wide range of models, including various particle systems, directed polymers, and interacting diffusions. To analyze large-scale random fluctuations for models in the so-called $1+1$ KPZ universality class, it is crucial to have good understandings of the invariant measures of the underlying Markov process, which in a sense can be viewed as (a discretization of ) the stochastic Burgers equation. 

The Burgers equation preserves the spatial mean, corresponding to the $\theta$ parameter in our case. Thus we use $\theta$ to parameterize the family of stationary solutions. One of our goals is to investigate the joint distributions of the coupled stationary solutions  $\{\mathscr{U}_\theta(\cdot,\cdot)\}_{\theta}$. In this paper, we consider the periodic setting. 

One of our surprising findings is that the derivative of $\mathscr{U}_\theta$ with respect to $\theta$ is related to the steady state of a diffusion in a random environment, described by the  SDE
\[
dY_t=-\mathscr{U}_\theta(t,Y_t)dt+dB_t.
\] This is an interesting diffusion process for at least two reasons:
\begin{enumerate}
    \item Suppose we denote the solution to the corresponding KPZ equation by $\mathscr{H}_\theta$ and view it as a randomly growing interface. Then the drift for the diffusion is $-\nabla \mathscr{H}_\theta$. Thus it is very natural to view  the process $Y_t$ as describing a particle sliding on the evolving interface, subject to some additional thermal noise. This is a model that has received a lot of attentions in the physics literature, see e.g.\ \cite{slider} and the references cited there, and is sometimes referred to as a ``passive slider''.
    \item The standard (deterministic, inviscid) Burgers equation takes the form \[\partial_t u+u\partial_x u=0,\] and the solution can be interpreted as the speed of the characteristics. With our  notation, it is actually ``$-\mathscr{U}_\theta$'' that plays the  role of the velocity, since the nonlinear term is on the r.h.s.\  rather than the l.h.s.\ of \eqref{e.burgers} by our convention. Therefore, we can view $Y_t$ as the position of a particle tracing the characteristics of the Burgers flow while being subjected to an independent thermal noise. If we consider  the asymmetric simple exclusion process (ASEP), which is a microscopic version of the Burgers equation, then a second class particle, which represents
        a microscopic characteristic, plays a similar role as the diffusion process. In the white noise setting, it is even likely that, with the weakly asymmetric scaling considered in \cite{BG97}, the second class particle process converges to the above diffusion, see e.g.\ \cite{BQS11}. Therefore, our result seems to be indicating an interesting connection between the coupled invariant measures and the ``shock profile'' observed by the second class particle. \blue{This observation is further developed in the subsequent work \cite{Dun24} by the first-named author, in which the inviscid Burgers equation with Poisson forcing is considered.} \blue{It is also clear if we draw an analogy between the stochastic Burgers equation and ASEP: the second class particles are the only differences between two coupled exclusion processes with different densities, and this elementary fact is essentially a discrete version of \eqref{e.mainidentity}, if we treat $\tilde{g}_\theta$ as the ``density'' of second class particles.} See Remark~\ref{r.shock} for a connection to another notion of ``shock profile.'' It is worth mentioning \blue{the recent work \cite{Bat23} where the ``environment seen from particle'' was studied in the first passage percolation}, and also  \cite{MSZ21} in which a result similar to Theorem~\ref{t.envi} was established for the second class particle, namely, the convergence of the environment seen from the second class particle (in the whole line setting).
\end{enumerate}

The study of coupled Busemann functions indexed by different slopes $\theta$ has seen important progress   during the past years. For a few solvable models in the $1+1$ KPZ universality class, one can obtain some explicit information on the joint distributions of the Busemann functions, see \cite{BFS23,BSS22,BSS221,BSS23,FS20,GRSS23,JRS22,SS23} and the references cited there.  One of the  focuses has been on studying the discontinuity of the Busemann function in the variable $\theta$, after fixing a realization of the random environment. \blue{This} turns out to be closely related to questions on the non-uniqueness of geodesics and the failure of the one-force-one-solution principle in a quenched sense. Compared to the aforementioned works, there is a fundamental difference in our periodic setting: by the moment bounds on $\tilde{g}_\theta$ obtained in Theorem~\ref{t.mainth} and an application of the Kolmogorov continuity theorem, we conclude that there exists a continuous version of the process $\{\mathscr{U}_\theta(t,x)\}_{\theta,t,x}$. In the case of a spacetime white noise, the same result holds for the Busemann function, in light of Theorem~\ref{t.white}. This is in sharp contrast with the non-periodic case. Very recently, the same problem was studied in  \cite{GRSS23} for the KPZ equation with a spacetime white noise in the non-periodic setting, and it was shown that, almost surely, \blue{the set of discontinuities} of the Busemann function in the variable $\theta$   is countably infinite and dense in $\R$. The difference seems to come from a stronger integrability in the compact setting, which can be seen from the expression on the r.h.s.\ of \eqref{e.lawBB}. What this implies about the geodesics and the one-force-one-solution principle is worth further study.

The continuity of $\mathscr{U}_\theta$ in $\theta$ on the torus can also be understood from the perspective of the preservation of the mean $\theta$ by the Burgers dynamics. On the torus, the identity $\int_{\bT}\mathscr{U}_\theta(t,x)dx=\theta$ should hold at least simultaneously for all $\theta\in\bQ$, for almost every realization of the noise. This eliminates the possibility of discontinuities of $\mathscr{U}_\theta$ in $\theta$: a discontinuity would introduce a jump in the process $\theta\mapsto \int_{\bT}\mathscr{U}_\theta(t,x)dx$, making it impossible for this process to be equal to $\theta$ for all $\theta\in\bQ$. The same argument does not work on $\R$, where the (quenched) mean should be defined as the limit of $\frac{1}{2L}\int_{-L}^{L} \mathscr{U}_\theta(t,x)dx$ as $L\to\infty$, so the difference of two solutions across a discontinuity can have zero mean  without being zero, as a result of the factor $L^{-1}$ as $L\to\infty$.

One of the main challenges in this paper is to show the existence of global solutions to the Fokker-Planck equation \eqref{e.fk1}. Such problems are the key in the study of  a passive tracer in turbulent transport. For ``compressible'' environments, whether the process of the environment seen from the particle, or, put it in different words, the Lagrangian velocity process, has an invariant measure, and whether it converges to the invariant measure is a fundamental question concerning the long time behaviors of the tracer. The problem is equivalent to studying the convergence of the Fokker-Planck equation of the form \eqref{e.fk1}, starting from the constant $1$. As it will become clear later, the   equation actually describes the evolution of the Radon-Nikodym derivative of the Lagrangian velocity process with respect to the Eulerian velocity. Even in the periodic setting, with a time-dependent drift that satisfies very good mixing properties, there is no standard recipe to prove such a result. A few  existing results are for the drifts modeled by a Gaussian process that decorrelates sufficiently rapidly, see e.g.\ \cite{DG22,tomasz1,tomasz3,tomasz2}. Our approach makes use of the special structure of the drift, as it solves the Burgers equation which is related to the KPZ and stochastic heat equation, through the Hopf-Cole transformation.

    \begin{remark}
        The paper \cite{DR21} on viscous shock solutions to the stochastic Burgers equation yields some analogies to the present work. That paper considered shock solutions to the stochastic Burgers equation connecting a ``top'' and ``bottom'' solution---the latter solutions being two ordered stationary solutions with different means. It was shown in \cite{DR21} that, at stationarity, the Radon--Nikodym derivative of the top and bottom solutions ``seen from a shock'' is proportional to the difference to the top and bottom solutions. Theorem~\ref{t.envi}  in the present work tells us that for the particle satisfying \eqref{e.defYt}, the Radon--Nikodym derivative of the environment seen from the particle is $\tilde g_\theta$, which is an infinitesimal limit of the difference between the top and bottom solutions.

        More precisely, the ``position'' of a shock connecting top and bottom solutions $u_{\theta_2}$ and $u_{\theta_1}$ solves the ODE
\begin{equation}\label{e.bode}
    \partial_tb_t = -\frac12\left(\partial_x\log(u_{\theta_2}-u_{\theta_1})+u_{\theta_2}+u_{\theta_1}\right)(t,b_t).
\end{equation}
This expression comes from \cite[(1.6--1.8)]{DR21}, with some signs changed because we use a different sign on the nonlinearity in \eqref{e.burgers}. 
Letting $\theta_1 = \theta-\eps$ and $\theta_2 = \theta+\eps$ and taking $\eps\searrow 0$, 
we get the limiting ODE
\begin{equation}
    \partial_t b_t = - u_\theta(t,b_t)-\frac12 \partial_x\log \tilde g_\theta(t,b_t).\label{e.limitODE}
\end{equation}
At stationarity, the environment seen from this limiting $b_t$ has Radon--Nikodym derivative $\tilde g_\theta$, just like the situation in Theorem~\ref{t.envi}.
On the other hand, the ODE \eqref{e.limitODE} is not the same as the diffusion \eqref{e.defYt}, although it does share the $-u_\theta$ term.
\label{r.shock}
    \end{remark}

\subsection{Organization of the paper}
The paper is organized as follows. In Section~\ref{s.bfp}, we derive a random Fokker-Planck equation  from the stochastic Burgers equation, through computing the derivative of $u_\theta$ in the variable $\theta$. A one-force-one-solution principle is established in Section~\ref{s.ofos} for the Fokker-Planck equation, which consists of the main technical part of the paper. In Section~\ref{s.proofmain}, we complete the proof of Theorems~\ref{t.mainth} and \ref{t.envi}. In Section~\ref{s.white}, we study the case of a spacetime white noise, and the argument there is mostly based on \cite{GK23}. Some basic facts on the stochastic Burgers equation are left in the appendix.

\subsection*{Notations and conventions.} 
\begin{enumerate} \item Throughout the paper, we denote by $\E$ expectation with respect to an auxiliary Brownian motion, and $\EE$ as the expectation with respect to the random noise $\xi$.

\item We denote $q_t(x)=\tfrac{1}{\sqrt{2\pi t}} e^{-x^2/(2t)}$ the standard heat kernel on $\R$, and we also define  
\blue{\[
        G_t(x)=\sum_{n\in\Z} q_t(x+n) \quad\text{and}\quad Q_t(x)=\sum_{n\in\Z} |\partial_xq_t(x+n)|.
\]}

\item \blue{$\mathcal{Z}_{t,s}(x,y), \mathcal{G}_{t,s}(x,y)$ are the propagators of the stochastic heat equation on $\R_+\times \R$ and $\R_+\times \bT$ respectively; see the precise definitions in \eqref{e.defZtsxy} and \eqref{e.defGtsxy}.}

\item \blue{For any $\theta\in\R$, we define the sheared noise 
\[
\xi^\theta(t,x)=\xi(t,x-\theta t), \quad\quad (t,x)\in\R^2.
\]}
\item When there is no confusion, we write $\sum_n=\sum_{n\in\Z}$.

\item \blue{${\mathcal M}_1(\bT)$ is the set of all Borel probability
measures  on $\bT$.}

\item We use $\|\cdot\|_p$ to denote the $L^p(\Omega)$ norm. 
\item \blue{In all the estimates we derive, the constants $c,C>0$ depend on $R(\cdot)$ and $p\in [1,\infty)$, which may change from line to line.}

\item Throughout the paper, we will identify functions on $\bT$ with $1$-periodic functions on $\R$, and sometimes we do not distinguish between them.

 \item \blue{The case of a smooth noise is analyzed in Sections 2--4, and the white noise case is studied in Section 5. The assumptions on the noise are implicitly imposed in the corresponding sections.}
\end{enumerate}

\subsection*{Acknowledgement.}
Y.G. would like to acknowledge the hospitality of Columbia University in fall 2023, where part of the work was done, and he would like to thank Ivan Corwin and Evan Sorensen for multiple discussions on the subject.
 Y.G. was partially supported by the NSF through DMS-2203014.   We would like to thank  the anonymous referee for sharing their expertise on the subject and for a very careful reading of the manuscript and many helpful suggestions and comments.

\section{From Burgers to Fokker-Planck}
\label{s.bfp}

In this section, we will start from the Burgers equation \eqref{e.burgers} and analyze how its solution depends on the parameter $\theta$, which only appears explicitly in the initial data. One could try to formally take the derivative of $\theta$ on both sides of \eqref{e.burgers}, which would result in a Fokker-Planck equation. The goal of this section is to justify this procedure rigorously. 
\blue{Throughout this section, we consider the case when $\xi$ is spatially smooth.}

\blue{We start from the Hopf-Cole transformation:
\[
u_\theta(t,x;-T)=\partial_x \log Z_\theta(t,x;-T),
\]
with $Z_\theta$ solving the stochastic heat equation 
\[
\partial_t Z_\theta =\frac12\Delta Z_\theta+Z_\theta \xi,\quad\quad Z_\theta(0,x;-T)=e^{\theta x}.
\]
Here the product between $Z_\theta$ and $\xi$ is interpreted in the It\^o-Walsh sense, and the solution admits the following Feynman--Kac representation \cite{BC95}:
\[
Z_\theta(t,x;-T)=\E\big[e^{\int_0^{t+T}\xi(t-s,x+B_s)ds-\frac12R(0)(t+T)}e^{\theta(x+B_{t+T})}\big].
\]
\blue{Here the $\E$ denotes expectation with respect to an auxiliary Brownian motion $B$.}
From the above expression, we can see that, almost surely, for each $t\geq -T$ and $x\in\bT$,  $u_{\theta}(t,x;-T)$ is differentiable in $\theta$; see a proof in Lemma~\ref{l.jointsmooth}.}

%

So our starting point is the following identity: for any $\lambda,t\in\R$ and $x\in\bT$, 
\[
u_{\lambda}(t,x;-T)-u_{0}(t,x;-T)=\int_0^\lambda \partial_\theta u_{\theta}(t,x;-T)d\theta,
\]
and the goal is to send $T\to\infty$ on both sides of the above identity and show the convergence in a proper sense. The convergence of the l.h.s.\ is a classical result of the one-force-one-solution principle of the Burgers equation with a periodic forcing; see e.g.\ \cite{sinai}. We will provide a proof in Appendix~\ref{s.aburgers} for the convenience of the reader. So the main challenge is to study the r.h.s.\  To this end, define 
\[
g_{\theta}(t,x;-T)=\partial_\theta u_{\theta}(t,x;-T).
\]

The main result of this section is to show that $g_\theta$ solves a Fokker-Planck equation:
\begin{proposition}\label{p.fpg}
For any $\theta\in\R$, $g_\theta$ solves the problem
\begin{equation}\label{e.fpg}
\begin{aligned}
\partial_t 
g_{\theta}(t,x;-T)&=\frac12\Delta g_{\theta}(t,x;-T)+\nabla(u_{\theta}g_{\theta})(t,x;-T), \quad\quad t>-T,x\in\R,\\
g_{\theta}(-T,x;-T)&=1.
\end{aligned}
\end{equation}
\end{proposition}

\begin{proof}
    One can directly take the derivative with respect to $\theta$, in the mild formulation of \eqref{e.burgers}, to obtain \eqref{e.fpg}. For our purpose (and the rest of the paper), we provide a different proof. 

First, through the Hopf-Cole transformation, we can write 
\[
u_\theta(t,x;-T)=\partial_x \log Z_{\theta}(t,x;-T),
\] with $Z_\theta$ solving 
\[
\begin{aligned}
\partial_t Z_\theta(t,x;-T)&=\frac12\Delta Z_\theta(t,x;-T) +Z_\theta(t,x;-T)  \xi(t,x),  \quad\quad t>-T,x\in\R,\\
Z_{\theta}(-T,x;-T)&=e^{\theta x}.
\end{aligned}
\]
\blue{For fixed $t,T$, the function $Z_\theta(t,x;-T)$ is smooth in the $(\theta,x)$ variable; see Lemma~\ref{l.jointsmooth} in the appendix.}
Therefore, one can interchange the order of partial derivatives and obtain
\begin{equation}\label{e.regZ}
g_\theta(t,x;-T)=\partial_{\theta}u_\theta(t,x;-T)=\partial_x \partial_\theta\log Z_{\theta}(t,x;-T).
\end{equation}
It turns out that $\partial_{\theta}\log Z_{\theta}$ is related to the quenched mean of a directed polymer in a random environment: by the Feynman-Kac formula, we have 
\[
Z_{\theta}(t,x;-T)=\E \big[e^{\int_0^{t+T} \xi(t-s,x+B_s)ds-\frac12R(0)(t+T)}e^{\theta (x+B_{t+T})}\big],
\]
where $\E$ is the expectation on the standard Brownian motion $B$. This leads to 
\begin{equation}\label{e.reZphi}
\begin{aligned}
\partial_\theta \log Z_{\theta}(t,x;-T)
&=\frac{\E \big[e^{\int_0^{t+T} \xi(t-s,x+B_s)ds}e^{\theta (x+B_{t+T})}(x+B_{t+T})\big]}{\E \big[e^{\int_0^{t+T} \xi(t-s,x+B_s)ds}e^{\theta (x+B_{t+T})}\big]}\\
&=x+\phi_{\theta}(t,x;-T),
\end{aligned}
\end{equation}
where 
\[
\blue{\phi_\theta(t,x;-T)}:= \frac{\E\big[e^{\int_0^{t+T}\xi(t-s,x+B_s)ds}e^{\theta(x+B_{t+T})}B_{t+T}\big]}{\E\big[e^{\int_0^{t+T}\xi(t-s,x+B_s)ds}e^{\theta(x+B_{t+T})}\big]}\] is the average displacement of the polymer endpoint. Applying Lemma~\ref{l.eqphi} below, we know $\phi_\theta$ solves 
\begin{equation}\label{e.eqphi}
\begin{aligned}
\partial_t\phi_\theta(t,x;-T)&=\frac12\Delta\phi_\theta(t,x;-T)+u_{\theta}(t,x;-T)[\nabla\phi_{\theta}(t,x;-T)+1], \quad\quad t>-T,x\in\R,\\
\phi_{\theta}(-T,x;-T)&=0.
\end{aligned}
\end{equation}
By \eqref{e.regZ} and \eqref{e.reZphi}, we have $g_\theta=1+\partial_x \phi_\theta$, which can be checked to solve \eqref{e.fpg}.
\end{proof}

\begin{lemma}\label{l.eqphi}
The function $\phi_{\theta}(t,x;-T)$ solves \eqref{e.eqphi}.
\end{lemma}

Before presenting the proof of Lemma~\ref{l.eqphi}, we introduce another notation that will be used throughout the paper.  Fix any realization of $\xi$ and parameters $t>-T,x,\theta\in\R$. We define the (quenched) polymer measure $\hat{\Pb}_{t,-T,x,\theta}$ on 
\[
C_x[0,t+T]=\{f\in C[0,t+T]:f(0)=x\}
\] such that for any bounded function $F:C_x[0,t+T]\to\R$, we have 
\begin{equation}\label{e.defpolymer}
\int_{C_x[0,t+T]}F(\omega)\hat{\Pb}_{t,-T,x,\theta}(d\omega)=\frac{\E \big[e^{\int_0^{t+T} \xi(t-s,x+B_s)ds}e^{\theta (x+B_{t+T})}F(x+B_{\cdot})\big]}{\E \big[e^{\int_0^{t+T} \xi(t-s,x+B_s)ds}e^{\theta (x+B_{t+T})}\big]}.
\end{equation}
Thus, the polymer measure $\hat{\Pb}_{t,-T,x,\theta}$ can be viewed as the Wiener measure on $C_x[0,t+T]$ tilted by the factor $e^{\int_0^{t+T} \xi(t-s,x+B_s)ds}e^{\theta (x+B_{t+T})}$. 

\begin{proof}[Proof of Lemma~\ref{l.eqphi}]
To simplify the notation in the proof, we assume that $T=0$ and consider
\begin{equation}\label{e.4221}
\phi_{\theta}(t,x;0)=\frac{\E \big[e^{\int_0^{t} \xi(t-s,x+B_s)ds}e^{\theta (x+B_{t})}B_{t}\big]}{\E \big[e^{\int_0^{t} \xi(t-s,x+B_s)ds}e^{\theta (x+B_{t})}\big]}.
\end{equation}
 
As is well-known, and can be seen through the Girsanov's theorem, the polymer measure is also the measure of 
a diffusion in a random environment: for each realization of $\xi$ and parameters $t>0,x,\theta\in\R$, we know from \cite[Lemma 4.2]{GK23} that the polymer measure $\hat{\Pb}_{t,0,x,\theta}$ is the same as the law of solution $\{X_s\}_{s\in[0,t]}$ to the SDE
\begin{equation}\label{e.4222}
dX_s=u_{\theta}(t-s,X_s;0)ds+dB_s, \quad\quad X_0=x.
\end{equation}
\blue{Thus, $B_t$ under the quenched polymer measure has the same law as $X_t-x$ (in the expression of \eqref{e.4221}, $\{x+B_s\}_{s\in[0,t]}$ is the polymer path, which has the same law as  $\{X_s\}_{s\in[0,t]}$.)
This implies that 
\[
\phi_{\theta}(t,x;0)=\E X_t-x,
\] because, from \eqref{e.4221}, we know that $\phi_\theta(t,x;0)$ is the quenched expectation of $B_t$.} (We emphasize again that the expectation $\E$ is taken only over $B$.) \blue{By the SDE \eqref{e.4222}, one can write 
\begin{equation}\label{e.881}
\E X_t-x=\E \int_0^t u_\theta(t-s,X_s;0)ds.
\end{equation}}

On the other hand,   suppose we consider the solution $\psi_\theta = \psi_\theta(t,x)$ to 
\[
\partial_t\psi_\theta=\frac12\Delta\psi_\theta+u_\theta \nabla\psi_\theta+u_\theta, \quad\quad t>0,x\in\R,
\]
with zero initial data $\psi_\theta(0,x)=0$. \blue{Through an application of It\^o's formula  to 
\[
Y_s=\psi_\theta(t-s,X_s),
\] for each realization of $u_\theta$: 
\[
dY_s=-\partial_t \psi_\theta(t-s,X_s)ds+\nabla \psi_\theta(t-s,X_s)dX_s+\frac12\Delta \psi_\theta(t-s,X_s)ds,
\]
we obtain that (with $X_s$ given by the solution to \eqref{e.4222})
\[
-\psi_\theta(t,x)=Y_t-Y_0=\int_0^t\nabla\psi(t-s,X_s)dB_s-\int_0^tu_\theta(t-s,X_s;0)ds.
\]}
Taking the expectation over $B$ in the above equation, we obtain
\[
\psi_\theta(t,x)=\E \int_0^t u_\theta(t-s,X_s;0)ds=\E X_t-x,
\]
where \eqref{e.881} is used in the second ``$=$''. This implies $\phi_\theta(t,x;0)=\psi_\theta(t,x)$ and  completes the proof.
\blue{Note that the above expression is the probabilistic representation of the solution $\psi$: on the intuitive level, the generator of the underlying Markov process $X$ is $\frac12\Delta +u_\theta\nabla$, and when  integrating the drift $u_\theta$ along the trajectory of the process itself, one obtains the displacement of the process.} 
\end{proof}

Now we note that \eqref{e.fpg} is a Fokker-Planck equation starting from constant $1$, so for each $t>-T$, the random field $\{g_{\theta}(t,x;-T)\}_{x\in\R}$ is positive, stationary, and has \blue{expectation $1$ (see \eqref{e.fkex1} below)}. Since $u_\theta$ is periodic in space, we know that $g_{\theta}$ is also periodic in space and 
\[
\int_0^1 g_{\theta}(t,x;-T)dx=1.
\]
Consider a diffusion on the unit torus $\bT$ with the random drift $-u_\theta$ 
\begin{equation}\label{e.sdeXs}
d\cX_s=-u_\theta(s-T,\cX_s;-T)ds+dB_s, \quad\quad s>0,
\end{equation}
with $\cX_0$ sampled from the uniform measure on $\bT$. It is clear that the quenched density of $\cX_s$ is given by $g_{\theta}(s-T,\cdot\,;\,-T)$. In the following sections, we will show that $g_\theta(t,x;-T)$ converges as $T\to\infty$, and it turns out that the limit can be interpreted as the Radon-Nikodym derivative which is related to a similar diffusion in a random environment.

%

\section{A one-force-one-solution principle for the Fokker-Planck equation}
\blue{In this section, we continue to consider the case when $\xi$ is smooth in space.}
\label{s.ofos}
From the previous section, we know that \[
u_{\lambda}(t,x;-T)-u_{0}(t,x;-T)=\int_0^\lambda g_{\theta}(t,x;-T)d\theta,
\]
with $g_\theta$ solving
\begin{equation}\label{e.fk3}
\begin{aligned}
\partial_t 
g_{\theta}(t,x;-T)=\frac12\Delta g_{\theta}(t,x;-T)+\nabla(u_{\theta}g_{\theta})(t,x;-T), \quad\quad t>-T,x\in\bT,
\end{aligned}
\end{equation}
with initial data $g_\theta(-T,x;-T)=1$.
The goal of this section is to show that, for any fixed $(t,x)$, the random variable  $g_{\theta}(t,x;-T)$ converges strongly as $T\to\infty$, \blue{as its dependence on the ``faraway'' random environments $u_\theta(s,\cdot;-T)$ diminishes as $t-s$ increases. This is sometimes referred to as the ``one-force-one-solution'' principle in random dynamical systems, which roughly says that the stationary solution is a measurable function of the forcing.} We will consider  more general initial data 
\[
g_\theta(-T,x;-T)=g_{\ic}(x).
\] We assume that $g_{\ic}$ is $1$-periodic and that $\int_{\bT}g_{\ic}(x)dx=1$.


Here is the main result of this section:

\begin{theorem}\label{t.ofos}
Fix any $(t,x)$, as $T\to\infty$,  $g_{\theta}(t,x;-T)$ converges in $L^p(\Omega)$ for any $p\in[1,\infty)$. Denoting the limit by $\tilde{g}_\theta(t,x)$, which takes the explicit form \eqref{e.explicitg}, we have 
\[
\EE |g_{\theta}(t,x;-T)-\tilde{g}_\theta(t,x)|^p  \leq Ce^{-c(t+T)}
\]
for some constants $C<\infty$ and $c>0$ depending only on $p$ and $R(\cdot)$.
\end{theorem}

As discussed in the introduction, to prove the result in Theorem~\ref{t.ofos} for general drifts is a difficult task, even in the periodic setting. As a matter of fact, we do not know any other example of a non-Gaussian drift in which such a result is proved. In our case with the drift given by the solution to the Burgers equation, we will make full use of the Hopf-Cole transformation and the connection to the stochastic heat equation (hence to the directed polymer through the Feynman-Kac representation). 

Throughout the rest of the section, we will assume without loss of generality that $t=0$.

\subsection{Probabilistic representation}
To prove Theorem~\ref{t.ofos}, we start from a probabilistic representation, which relies on the Feynman-Kac representation of the solution to the stochastic heat equation. A somewhat surprising fact is that, through the Hopf-Cole transformation and the probabilistic representation, the one-force-one-solution principle for $g_\theta$, as stated in Theorem~\ref{t.ofos}, comes from the one-force-one-solution principle for $u_\theta$.

First, as in the proof of Lemma~\ref{l.eqphi}, we define the diffusion $X$ through the SDE (note that $t=0$ in this section)
\begin{equation}\label{e.Xnew}
dX_s=u_{\theta}(-s,X_s;-T)dt+dB_s, \quad\quad X_0=x,
\end{equation}
and $\phi_\theta$ as the solution to
\begin{equation}\label{e.eqphithetanew}
\partial_t\phi_\theta(t,x;-T)=\frac12\Delta\phi_\theta(t,x;-T)+u_{\theta}(t,x;-T)[\nabla\phi_{\theta}(t,x;-T)+1].
\end{equation}
Since the initial data for $g_\theta$ is $g_\theta(-T,x;-T)=g_{\ic}(x)$, to maintain the identity
\[
g_\theta=1+\partial_x\phi_\theta,\] we assume that $\phi_\theta(-T,x;-T)=\phi_{\ic}(x)$ with $\phi_{\ic}'(x)=g_{\ic}(x)-1$. Note that the function $\phi_{\ic}$ is $1$-periodic because of the assumption of $\int_0^1 g_{\ic}(x)dx=1$.

Following the proof of Lemma~\ref{l.eqphi} verbatim, we obtain the following probabilistic representation of the solution to \eqref{e.eqphithetanew}:
\[
\phi_{\theta}(0,x;-T)=\E\phi_{\ic}(X_T)+\E X_T-x.
\]
As discussed in the proof of Lemma~\ref{l.eqphi}, we know that the (quenched) law of the process $\{X_s\}_{s\in[0,T]}$ is the (quenched) polymer measure $\hat{\Pb}_{0,-T,x,\theta}$. It turns out that by the polymer Gibbs measure formulation,  one can write the density of $X_s$ \blue{explicitly}.

Denote $\cZ_{t,s}(x,y)$ as the propagator for the stochastic heat equation, so for any $(s,y)\in\R^2$, 
\begin{equation}\label{e.defZtsxy}
\begin{aligned}
&\partial_t \cZ_{t,s}(x,y)=\frac12\Delta_x \cZ_{t,s}(x,y)+\xi(t,x)\cZ_{t,s}(x,y), \quad\quad t>s,x\in\R,\\
&\cZ_{s,s}(x,y)=\delta_y(x).
\end{aligned}
\end{equation}
\blue{In our case of a smooth noise,  
it admits the Feynman--Kac representation (with the expectation $\E$ below only on the standard Brownian motion $B$)
\begin{equation}\label{e.4231}
\begin{aligned}
\cZ_{t,s}(x,y)&=\E \left[e^{\int_0^{t-s}\xi(t-\ell,x+B_\ell)d\ell-\frac12R(0)(t-s)}\delta_y(x+B_{t-s})\right]\\
&=q_{t-s}(x-y)\E\left[e^{\int_0^{t-s}\xi(t-\ell,x+B_\ell)d\ell-\frac12R(0)(t-s)}\,|\, x+B_{t-s}=y\right],
\end{aligned}
\end{equation}
where $q$ is the standard Gaussian kernel.}

Let $\rho_{\theta}(0,x;s,\cdot;-T)$ be the quenched density of $X_s$. By the Gibbs measure formulation in \eqref{e.defpolymer}, we have 
\begin{lemma}\label{l.exrho}
For any $y\in\R$ and $s\in(0,T]$, we have 
\[
\begin{aligned}
\rho_{\theta}(0,x;s,y;-T)&=\frac{\E[e^{\int_0^T \xi(-t,x+B_t)dt}e^{\theta (x+B_T)}\delta_y(x+B_s) ]}{\E[e^{\int_0^T \xi(-t,x+B_t)dt}e^{\theta(x+ B_T)}]}\\
&=\frac{\cZ_{0,-s}(x,y)\int_{\R} \cZ_{-s,-T}(y,w)e^{\theta w}dw}{\int_{\R}\cZ_{0,-T}(x,w)e^{\theta w}dw}.
\end{aligned}
\]
\end{lemma}

\begin{proof}
The result comes directly from the definition of the polymer measure in \eqref{e.defpolymer} and the Feynman-Kac formula. 
\end{proof}

Using the quenched density, one can express $\phi_\theta$ as 
\[
\begin{aligned}
&\phi_{\theta}(0,x;-T)\\
&=\E\phi_{\ic}(X_T)+\E X_T-x=\E\phi_{\ic}(X_T)+\E \int_0^T  u_{\theta}(-s,X_s;-T)ds\\
&=\int_{\R} \phi_{\ic}(y)\rho_{\theta}(0,x;T,y;-T)dy+\int_0^T\int_{\R} u_{\theta}(-s,y;-T)\rho_{\theta}(0,x;s,y;-T)dyds.
\end{aligned}
\]
Note that in the above expression, we have periodically extended $\phi_{\ic}(\cdot)$ and $u_{\theta}(-s,\cdot;-T)$ so that they are defined on $\R$. By the fact that $g_\theta=1+\partial_x\phi_\theta$, we have 
\[
\begin{aligned}
    g_{\theta}(0,x;-T)=1&+\int_{\R} \phi_{\ic}(y)\partial_x\rho_{\theta}(0,x;T,y;-T)dy\\
&+\int_0^T\int_{\R} u_{\theta}(-s,y;-T)\partial_x\rho_{\theta}(0,x;s,y;-T)dyds.
\end{aligned}
\]
Therefore, to prove Theorem~\ref{t.ofos}, it suffices to study $\partial_x \rho_{\theta}(0,x;s,y;-T)$ and show that its dependence on the faraway environment is exponentially small. 

Since $\phi_{\ic}(\cdot)$ and $u_{\theta}(-s,\cdot;-T)$ are both $1$-periodic, we define the periodized version of $\rho_\theta$. That is,  we consider $X$ as a diffusion on the torus, and study its density on the torus. For any $y\in[0,1)$, define 
\begin{equation}\label{e.defrhobar}
\bar{\rho}_{\theta}(0,x;s,y;-T)=\sum_{n\in\Z}\rho_{\theta}(0,x;s,y+n;-T),
\end{equation}
so we can also write 
\begin{equation}\label{e.phibarrho}
\begin{aligned}
g_{\theta}(0,x;-T)=1&+\int_0^1 \phi_{\ic}(y)\partial_x\bar{\rho}_{\theta}(0,x;T,y;-T)dy\\
&+\int_0^T\int_0^1 u_{\theta}(-s,y;-T)\partial_x\bar{\rho}_{\theta}(0,x;s,y;-T)dyds.
\end{aligned}
\end{equation}
At this moment, whether the two integrals in \eqref{e.phibarrho} \blue{are well-defined or not} is unclear, and we will show later that this is indeed the case.

To show the convergence of $g_{\theta}(0,x;-T)$ as $T\to\infty$, the integrand in the above expression must be small for $s\gg1$. Since $u_\theta(-s,y;-T)$ is $O(1)$, one needs to show that $\partial_x\bar{\rho}_{\theta}(0,x;s,y;-T)$ is small for large $s$. The intuition is that, as we are on the torus, the polymer density mixes exponentially fast so that the mid-point density for $s\gg1$ does not depend much on the starting point $x$. As a result, a differentiation in $x$ leads to a small quantity. The following lemma derives another expression of $\partial_x\bar{\rho}_{\theta}(0,x;s,y;-T)$, through which we see at least formally that this is the case when $s$ is large.

\begin{lemma}\label{l.debarrho}
We have 
\[
\partial_x \bar{\rho}_{\theta}(0,x;s,y;-T)=\bar{\rho}_{\theta}(0,x;s,y;-T)[\U_\theta(0,x;-s,y)-u_\theta(0,x;-T)],
\]
with 
\begin{equation}\label{e.defU}
\begin{aligned}
\U_\theta(0,x;-s,y):=\partial_x \log \sum_{n\in\Z} \cZ_{0,-s}(x,y+n)e^{\theta n}.
\end{aligned}
\end{equation}
\end{lemma}

\begin{proof}
By Lemma~\ref{l.exrho}, we can write the density as
\[
\begin{aligned}
\rho_{\theta}(0,x;s,y+n;-T)&=\frac{\cZ_{0,-s}(x,y+n)\int_{\R} \cZ_{-s,-T}(y+n,w)e^{\theta w}dw}{\int_{\R}\cZ_{0,-T}(x,w)e^{\theta w}dw}\\
&=\frac{\cZ_{0,-s}(x,y+n)e^{\theta n}\int_{\R} \cZ_{-s,-T}(y,w)e^{\theta w}dw}{\int_{\R}\cZ_{0,-T}(x,w)e^{\theta w}dw}.
\end{aligned}
\]
In the second ``$=$'', we used the spatial periodicity of $\xi$, which implies that 
\[
\cZ_{-s,-T}(y+n,w)=\cZ_{-s,-T}(y,w-n).
\] 
Therefore, 
\[
\begin{aligned}
    \partial_x \log \sum_{n\in\Z} \rho_{\theta}(0,x;s,y+n;-T)&=\partial_x \log \sum_{n\in\Z} \cZ_{0,-s}(x,y+n)e^{\theta n}\\
&\qquad-\partial_x \log \int_{\R}\cZ_{0,-T}(x,w)e^{\theta w}dw,
\end{aligned}
\]
which completes the proof.
\end{proof}

One should view $\U_\theta(0,x;-s,y)$ as the solution to the stochastic Burgers equation with the initial data of ``mean'' $\theta$ at time $-s$. This can be seen from \eqref{e.defU} and the fact that, after the Hopf-Cole transformation, the initial data (at time $-s$) for the corresponding stochastic heat equation is $\sum_n e^{\theta n}\delta_{y+n}$. Thus, by the one-force-one-solution principle and the contraction of the Burgers equation on the torus, for $s\gg1$ and $T\gg1$, we expect $\U_\theta(0,x;-s,y)-u_{\theta}(0,x;-T)$ to be exponentially small in $s$.

Now, combining  \eqref{e.phibarrho} and Lemma~\ref{l.debarrho}, we have
\begin{equation}\label{e.7191}
\begin{aligned}
g_{\theta}&(0,x;-T)=1+\int_0^1 \phi_{\ic}(y)\bar{\rho}_{\theta}(0,x;T,y;-T)[\U_\theta(0,x;-T,y)-u_\theta(0,x;-T)]dy\\
&+\int_0^T\int_0^1u_{\theta}(-s,y;-T)\bar{\rho}_{\theta}(0,x;s,y;-T) [\U_\theta(0,x;-s,y)-u_\theta(0,x;-T)]dyds,
\end{aligned}
\end{equation}
which will be our starting point to prove Theorem~\ref{t.ofos}.

\subsection{Moment estimates}
In this section, we will prove several moment estimates on the terms appearing in \eqref{e.7191}, including the solution to the Burgers equation, the mid-point quenched density etc.

We introduce a few more notations. For any $t>s$ and $x,y\in\R$, define 
\begin{equation}\label{e.defGtsxy}
\G_{t,s}(x,y)=\sum_{n\in\Z} \cZ_{t,s}(x+n,y)=\sum_{n\in\Z} \cZ_{t,s}(x,y+n),
\end{equation} which should be viewed as the propagator of the stochastic heat equation on the torus, and let 
\[
\G_{t,s}(x,-)=\int_0^1 \G_{t,s}(x,y)dy.
\] 
\blue{Note that $\G_{t,s}(x,y)$ also admits a probabilistic representation, and this can be obtained in two ways: (i) using the above definition together with the probabilistic representation of $\cZ_{t,s}(x,y)$ given in \eqref{e.4231}; (ii) replacing the Brownian motion in the expression of \eqref{e.4231} with its projection on $\bT$.}
 We also introduce the following notations on the endpoint densities
of the ``forward'' and ``backward'' polymer path on a cylinder. 
Let ${\mathcal M}_1(\bT)$ be the set of all Borel probability
measures  on $\bT$. 
For
any $\nu\in{\mathcal M}_1(\bT)$ and $t>s$, we let 
\begin{equation}\label{e.forwardbackward}
\begin{aligned}
&\rhof(t,x;s,\nu)=\frac{\int_{\bT} \G_{t,s}(x,y)\nu(dy)}{\int_{\bT^2}\G_{t,s}(x',y)\nu(dy)dx'}, \quad x\in\bT,\\
&\rhob(t,\nu;s,y)=\frac{\int_{\bT}\G_{t,s}(x,y)\nu(dx)}{\int_{\bT^2}\G_{t,s}(x,y')\nu(dx)dy'},\quad y\in\bT.
\end{aligned}
\end{equation}
Here the subscripts ``$\mathrm{f}$'' and ``$\mathrm{b}$'' stand for ``forward'' and ``backward'' respectively. When $\nu$ is a Dirac measure on $\bT$, we   write
$$\rhof(t,x;s,y)=\rhof(t,x;s,\delta_y)\quad \mbox{and}
\quad \rhob(t,x;s,y)=\rhob(t,\delta_x;s,y).
$$

For any $\theta\in\R$, define 
\[
\xi^\theta(t,x)=\xi(t,x-\theta t), \quad\quad (t,x)\in\R^2.
\]
By the fact that $\xi$ is a Gaussian process that is white in time and stationary in space, we know that the two processes $\xi^\theta$ and $\xi$ have the same law. Let $\cZ^\theta,\G^\theta$ be the propagators of the SHE on the line and the torus respectively, driven by $\xi^\theta$.  

\blue{The following lemma helps to reduce all the moment estimates on $u_\theta,\U_\theta, \bar{\rho}_\theta$ to the $\theta=0$ case, because by the previous discussion we know that $\G^\theta$ has the same law as $\G=\G^0$.}

\begin{lemma}\label{l.reduction}
We have
\begin{equation}\label{e.9271}
\begin{aligned}
&u_\theta(0,x;-T)=\theta+\partial_x  \log \int_0^1 \G^\theta_{0,-T}(x,y)dy,\\
&\U_{\theta}(0,x;-s,y)=\theta+\partial_x  \log \G^\theta_{0,-s}(x,y-\theta s),\\
&\bar{\rho}_\theta(0,x;s,y;-T)
=\frac{\G_{0,-s}^\theta(x,y-\theta s)\G_{-s,-T}^\theta(y-\theta s,-)}{\G_{0,-T}^\theta(x,-)}.
\end{aligned}
\end{equation}
\end{lemma}
\begin{proof}
The proofs below all rely on the Feynman-Kac formula and the Cameron-Martin theorem.  For $u_\theta(0,x;-T)$, we have
\[
\begin{aligned}
u_\theta(0,x;-T)&=\partial_x \log  \E e^{\int_0^T\xi(-t,x+B_t)dt}e^{\theta (x+B_T)}\\
&=\theta+\partial_x \log  \E e^{\int_0^T\xi(-t,x+B_t)dt}e^{\theta B_T-\frac12\theta^2 T}\\
&=\theta+\partial_x \log \E e^{\int_0^T \xi(-t,\theta t+x+B_t)dt}\\
&=\theta+\partial_x \log \E e^{\int_0^T\xi^\theta(-t,x+B_t)dt-\frac12R(0)T}=\theta+\partial_x  \log \int_0^1 \G^\theta_{0,-T}(x,y)dy.
\end{aligned}
\]
 For the last ``$=$'', we used the fact that $\int_0^1 \G^\theta_{0,-T}(x,y)dy$ solves the stochastic heat equation driven by $\xi^\theta$, with the constant initial data $1$ at time $-T$. \blue{Here we recall that the function $R(\cdot)$ appeared in the above expression is the spatial covariance function of $\xi$ and $\xi^\theta$}.

Similarly, by \eqref{e.defU} we can write $\U_{\theta}(0,x;-s,y)$ as 
\[
\begin{aligned}
\U_{\theta}(0,x;-s,y)&=\partial_x \log  \sum_n \cZ_{0,-s}(x,y+n)e^{\theta n}\\
&=\partial_x \log \sum_n \E \big[e^{\int_0^s\xi(-t,x+B_t)dt}e^{\theta n}\delta_{y+n}(x+B_s)\big]\\
&=\partial_x \log \sum_n \E \big[e^{\int_0^s\xi(-t,x+B_t)dt}e^{\theta (x+B_s-y)}\delta_{y+n}(x+B_s)\big].
\end{aligned}
\]
\blue{In the last ``='', we used the definition of the  Dirac function so that $e^{\theta n}$ can be replaced by $e^{\theta(x+B_s-y)}$. More precisely, the expectations above can be written as the   expectations with respect to a Brownian bridge, multiplied by the corresponding transition density of the Brownian motion. In this way, one sees immediately that the last ``='' holds.} 
Now we apply the Cameron-Martin theorem again, the above expression is equal to
\[
\begin{aligned}
&\theta+\partial_x \log \sum_n \E \big[e^{\int_0^s\xi(-t,x+\theta t+B_t)dt}\delta_{y+n}(x+\theta s+B_s)\big]\\
&=\theta+\partial_x \log \E \big[e^{\int_0^s\xi^\theta(-t,x+B_t)dt-\frac12R(0)s}\sum_n\delta_{y-\theta s+n}(x+B_s)\big].
\end{aligned}
\]
Using $\G^\theta$, the expectation on $B$ in the above expression can be rewritten as 
\[
\E \big[e^{\int_0^s\xi^\theta(-t,x+B_t)dt-\frac12R(0)s}\sum_n\delta_{y-\theta s+n}(x+B_s)\big]= \G^\theta_{0,-s}(x,y-\theta s).
\]

For the mid-point   density,  recall that $\bar{\rho}_\theta(0,x;s,y;-T)=\sum_n \rho_\theta (0,x;s,y+n;-T)$, and
\[
\rho_\theta(0,x;s,y+n;-T)=\frac{\E[e^{\int_0^T \xi(-t,x+B_t)dt}e^{\theta (x+B_T)}\delta_{y+n}(x+B_s) ]}{\E [e^{\int_0^T \xi(-t,x+B_t)dt}e^{\theta(x+ B_T)}]}.
\]
By the same discussion as in the previous case, we have
\[
\begin{aligned}
\bar{\rho}_\theta(0,x;s,y;-T)&=\sum_n \frac{\cZ_{0,-s}^\theta(x,y+n-\theta s)\int_{\R} \cZ_{-s,-T}^\theta(y+n-\theta s,w)dw}{\int_{\R} \cZ_{0,-T}^\theta(x,w)dw}\\
&=\frac{\G_{0,-s}^\theta(x,y-\theta s)\G_{-s,-T}^\theta(y-\theta s,-)}{\G_{0,-T}^\theta(x,-)}.
\end{aligned}
\]
In the second ``$=$'', we used the fact that 
\[
\int_{\R} \cZ_{-s,-T}^\theta(y+n-\theta s,w)dw=\int_{\R} \cZ_{-s,-T}^\theta(y-\theta s,w-n)dw=\G_{-s,-T}^\theta(y-\theta s,-).
\]
The proof is complete.
\end{proof}

Next, we derive the moment estimates on  the solutions to the Burgers equation with constant initial data: 
\begin{lemma}\label{l.bdu}
\blue{For any $p\in [ 1,\infty)$, there exists a constant $C=C(p,R)$ such that} \[
\EE |u_{\theta}(t,0;0)|^p \leq C(|\theta|^p+1).
\]
\end{lemma}

\begin{proof}
For each fixed $t>0$, by Lemma~\ref{l.reduction}  we have 
\[
u_{\theta}(t,0;0) \stackrel{\text{law}}{=}\theta+u_0(t,0;0),
\]
so the result is a direct consequence of Lemma~\ref{l.mmbd} below.
\end{proof}

Recall that $G_t(x)=\sum_{n} (2\pi t)^{-1/2}\exp(-(x+n)^2/(2t))$ is the heat kernel on the torus.  For the mid-point quenched density, we have 
\begin{lemma}\label{l.bdrho}
\blue{For  $p\in [1,\infty)$, there exists $C=C(p,R)$ such that for all $T\geq10$, }
\[
\EE |\bar{\rho}_{\theta}(0,x;s,y;-T)|^p \leq C\big(1_{s\geq1}+1_{s<1}G_s(y-\theta s-x)^p\big).
\]
\end{lemma}

\begin{proof}
\blue{By Lemma~\ref{l.reduction}, we have 
\[
\bar{\rho}_\theta(0,x;s,y;-T)
=\frac{\G_{0,-s}^\theta(x,y-\theta s)\G_{-s,-T}^\theta(y-\theta s,-)}{\G_{0,-T}^\theta(x,-)}
\]
By the fact that $\G\stackrel{\text{law}}{=}\G^\theta$, we claim 
\begin{equation}\label{e.midpointdensity}
\{\bar{\rho}_\theta(0,x;s,y;-T)\}_{x,s,y,T}\stackrel{\text{law}}{=}\left\{\frac{\rhob(0,x;-s,y-\theta s)\rhof(-s,y-\theta s;-T,\mathrm{m})}{\int_0^1\rhob(0,x;-s,y'-\theta s)\rhof(-s,y'-\theta s;-T,\mathrm{m})dy'}\right\}_{x,s,y,T}.
\end{equation}
Here ``$\mathrm{m}$'' denotes the uniform measure on $\bT$.
To see why \eqref{e.midpointdensity} holds, it suffices to use the definitions of the forward and backward polymer endpoint densities  \eqref{e.forwardbackward} to rewrite it as 
\begin{align*}
    A&:=\frac{\rhob(0,x;-s,y-\theta s)\rhof(-s,y-\theta s;-T,\mathrm{m})}{\int_0^1\rhob(0,x;-s,y'-\theta s)\rhof(-s,y'-\theta s;-T,\mathrm{m})dy'}\\&=\frac{\G_{0,-s}(x,y-\theta s)\G_{-s,-T} (y-\theta s,-)}{\G_{0,-T} (x,-)}.
\end{align*}}
Before estimating $\EE A^p$, we record a few facts about polymer endpoint densities: \blue{for any $p\in[1,\infty)$, there exists a constant $C=C(p, R)$ such that} 
\begin{equation}\label{e.7221}
\begin{aligned}
&\|\rhob(0,x;-s,y-\theta s)\|_p \leq C(G_s(y-\theta s -x)1_{s<1}+1_{s\geq1}),\\
&\|\rhob(0,x;-s,y-\theta s)^{-1}\|_p \leq C(G_s(y-\theta s -x)^{-1}1_{s<1}+1_{s\geq1}),\\
&\|\sup_{y\in\bT}\rhof(-s,y-\theta s;-T,\mathrm{m})\|_p +\|\sup_{y\in\bT}\rhof(-s,y-\theta s;-T,\mathrm{m})^{-1}\|_p\leq C.\\
\end{aligned}
\end{equation}

These estimates are rather standard. The proof of the  bound for $s\geq1$ can be found e.g.\ in \cite[Proposition A.1]{GK23}, while the bound for $s<1$ can be proved using the Feynman-Kac formula and Jensen's inequality in this case. The difference between the estimates for $\rhob$ and $\rhof$ comes from the different initial conditions: for $\rhob$, we have $\delta_x$ as the starting distribution, while for $\rhof$, the initial distribution is $\mathrm{m}$ so the singularity has been  integrated out. \blue{For the convenience of readers, we leave the detailed proof in the appendix, see Lemma~\ref{l.estipolymer}.}

To estimate $\EE A^p$, if $s\geq1$, it suffices to apply  \eqref{e.7221} and the H\"older inequality. If $s<1$, we first bound $A$ by 
\[
A\leq \rhob(0,x;-s,y-\theta s)\rhof(-s,y-\theta s;-T,\mathrm{m})  (\inf_{y'\in\bT}\rhof(-s,y'-\theta s;-T,\mathrm{m}))^{-1},
\]
then apply \eqref{e.7221} and the H\"older inequality to complete the proof.
\end{proof}

Recall that $q_t(x)=(2\pi t)^{-1/2}e^{-x^2/(2t)}$. We introduce the notation 
\begin{equation}\label{e.defQ}
Q_t(x)=\sum_{n\in\Z} |\partial_xq_t(x+n)|, \quad\quad t>0,x\in\R.
\end{equation}
The following lemma provides a moment estimate for the solution to the Burgers equation  with  ``Dirac'' initial data:
\begin{lemma}\label{l.bdU}
    \blue{For any $p\in [1,\infty)$, there exists a constant $C=C(p,R)$ such that} 
\[
\EE |\U_{\theta}(0,x;-s,y)|^p\leq C\bigg(|\theta|^p+1_{s\geq 1}+1_{s<1}\big(1+\frac{Q_s(y-\theta s-x)^p}{G_s(y-\theta s-x)^p}\big)\bigg).
\]
\end{lemma}

\begin{proof}
By Lemma~\ref{l.reduction}, we have 
\[
\U_{\theta}(0,x;-s,y)\stackrel{\text{law}}{=}\theta+\partial_x  \log \G_{0,-s}(x,y-\theta s),
\]
so it remains to estimate the moments of the second term on the r.h.s. 

For $s\geq1$, the estimate comes from Lemma~\ref{l.mmbd} below. 

For $s<1$, we apply Cauchy-Schwarz to derive
\[
\begin{aligned}
&\EE |\partial_x  \log \G_{0,-s}(x,y-\theta s)|^p \\
&\leq \sqrt{ \EE |\partial_x \G_{0,-s}(x,y-\theta s)|^{2p}\EE \G_{0,-s}(x,y-\theta s)^{-2p}}.
\end{aligned}
\]
Applying Lemma~\ref{l.bddegreen} below,  we have 
\[
\begin{aligned}
&\EE \G_{0,-s}(x,y-\theta s)^{-2p} \leq C G_s(y-\theta s-x)^{-2p},\\
&\EE |\partial_x \G_{0,-s}(x,y-\theta s)|^{2p} \leq C\big(  Q_s(y-\theta s-x)^{2p}+G_s(y-\theta s-x)^{2p}\big).
\end{aligned}
\]
Combining the above two estimates,  we conclude that 
\[
\EE |\partial_x  \log \G_{0,-s}(x,y-\theta s)|^p  \leq C\big(1+\frac{Q_s(y-\theta s-x)^{p}}{G_s(y-\theta s-x)^p}\big),
\]
which completes the proof.
\end{proof}

\subsection{Exponential mixing of polymer and Burgers}
In this section, we prove a few stability results, which will be used to show the convergence of the r.h.s.\ of \eqref{e.7191}.

The first result is the stabilization of the polymer mid-point density: 
\begin{lemma}\label{l.mixpolymer}
    \blue{For any $p\in [1,\infty)$, there exists constants $C,c>0$ depending on $p,R(\cdot)$ such that} for any $T_2\geq T_1\geq 10$, we have 
\[
\begin{aligned}
  &\EE |\bar{\rho}_{\theta}(0,x;s,y;-T_1)-\bar{\rho}_{\theta}(0,x;s,y;-T_2)|^p \\
  &\leq Ce^{-c(T_1-s)}\big(1_{s\geq1}+1_{s<1} G_s(y-\theta s-x)^p\big).
\end{aligned}
\]
As a result, $\bar{\rho}_\theta(0,x;s,y;-T)$ converges in $L^p(\Omega)$ as $T\to\infty$.
\end{lemma}
This shows that for any fixed $s>0$, as $T\to\infty$, the mid-point density $\bar{\rho}_\theta(0,x;s,y;-T)$ converges. The proof utilizes the fast mixing of the endpoint density, proved in \cite[Theorem 2.3]{GK21}, see also \cite{rosati} for a different proof.

\begin{proof}[Proof of Lemma~\ref{l.mixpolymer}]
From \eqref{e.midpointdensity} (as a result of Lemma~\ref{l.reduction}), we have
\[
\{\bar{\rho}_\theta(0,x;s,y;-T)\}_{x,s,y,T}\stackrel{\text{law}}{=}\left\{\frac{\rhob(0,x;-s,y-\theta s)\rhof(-s,y-\theta s;-T,\mathrm{m})}{\int_0^1\rhob(0,x;-s,y'-\theta s)\rhof(-s,y'-\theta s;-T,\mathrm{m})dy'}\right\}_{x,s,y,T}.
\]
Thus, it suffices to prove the same result for the r.h.s.\ of the above display, which we denote by $A(T)$.

 First, we write $A(T_1)-A(T_2)=\mathrm{err}_1+\mathrm{err}_2$ with 
\[
\mathrm{err}_1=\frac{\rhob(0,x;-s,y-\theta s)[\rhof(-s,y-\theta s;-T_1,\mathrm{m})-\rhof(-s,y-\theta s;-T_2,\mathrm{m})]}{\int_0^1\rhob(0,x;-s,y'-\theta s)\rhof(-s,y'-\theta s;-T_1,\mathrm{m})dy'},
\]
and
\[
\begin{aligned}
\mathrm{err}_2=&\frac{\rhob(0,x;-s,y-\theta s)\rhof(-s,y-\theta s;-T_2,\mathrm{m})}{\prod_{j=1}^2\int_0^1\rhob(0,x;-s,y'-\theta s)\rhof(-s,y'-\theta s;-T_j,\mathrm{m})dy'}\\
&\times\int_0^1 \rhob(0,x;-s,y'-\theta s)[\rhof(-s,y'-\theta s;-T_2,\mathrm{m})-\rhof(-s,y'-\theta s;-T_1,\mathrm{m})]dy'.
\end{aligned}
\]
To estimate $\EE|\mathrm{err}_i|^p$ with $i=1,2$, besides \eqref{e.7221}, we also need the following estimate (\cite[Proposition A.1]{GK23}):
\[
\begin{aligned}
\sup_{y}\|\rhof(-s,y-\theta s;-T_1,\mathrm{m})-\rhof(-s,y-\theta s;-T_2,\mathrm{m})\|_p \leq Ce^{-c(T_1-s)}.
\end{aligned}
\]
Applying the H\"older inequality as in the proof of Lemma~\ref{l.bdrho}, we obtain the desired bound for $\EE|\mathrm{err}_i|^p$ for $i=1,2$. 
\end{proof}

The second goal of this section is to show that the solutions to the Burgers equation, started from different initial data, are exponentially close in time:
\begin{proposition}\label{p.burgersofos}
For any $p$, there exist constants $C<\infty$ and $c>0$ so that 
\begin{equation}\label{e.7211}
\sup_{x,y\in[0,1]}\EE|\U_\theta(0,x;-s,y)-u_\theta(0,x;-T)|^p \leq Ce^{-cs}, \quad\quad  1\leq s\leq T.
\end{equation}
We also have 
\begin{equation}\label{e.7212}
\EE |u_{\theta}(-s,y;-T_2)-u_{\theta}(-s,y;-T_1)|^p \leq C e^{-c(T_1-s)}, \quad\quad 0\leq s\leq T_1\leq T_2.
\end{equation}
\end{proposition}
The above proposition is essentially  the one-force-one-solution principle of the stochastic Burgers equation with a periodic forcing, which is well-known  given the classical result of Sinai \cite{sinai}. Since we could not find an exact reference for our purpose, we provide a proof in the appendix for the convenience of the readers.


\subsection{Proof of Theorem~\ref{t.ofos}}
Now we have all the ingredients that are needed to complete the proof of Theorem~\ref{t.ofos}. 

Fix $x\in\bT$ in this section. Recall from \eqref{e.7191} that 
\begin{equation}\label{e.exgtheta}
\begin{aligned}
g_{\theta}&(0,x;-T)=1+\int_0^1 \phi_{\ic}(y)\bar{\rho}_{\theta}(0,x;T,y;-T)[\U_\theta(0,x;-T,y)-u_\theta(0,x;-T)]dy\\
&+\int_0^T\int_0^1u_{\theta}(-s,y;-T)\bar{\rho}_{\theta}(0,x;s,y;-T) [\U_\theta(0,x;-s,y)-u_\theta(0,x;-T)]dyds.
\end{aligned}
\end{equation}
Let us first check that the integrals on the r.h.s.\ \blue{are} well-defined: take the second integral as an example, combining Lemmas~\ref{l.bdu}, \ref{l.bdrho} and \ref{l.bdU}, we have that for any $p\in [1,\infty)$, 
\begin{equation}\label{e.bdintegrand}
\begin{aligned}
&\|u_{\theta}(-s,y;-T)\bar{\rho}_{\theta}(0,x;s,y;-T) [\U_\theta(0,x;-s,y)-u_\theta(0,x;-T)]\|_p\\
&\leq C \big(1_{s\geq1}+1_{s<1}G_s(y-\theta s-x)\big)\big(1+1_{s<1}\tfrac{Q_s(y-\theta s-x)}{G_s(y-\theta s-x)})\big)
\end{aligned}
\end{equation}
\blue{We claim the r.h.s.\ of \eqref{e.bdintegrand} belongs to $L^1([0,T]\times [0,1])$. To see why, we  consider the range of $s<1$ and prove that $G_s(y)$ and $Q_s(y)$ are locally integrable. By the definition of $G_s$ and $Q_s$, the integrability is a consequence of  the facts that $\int_{\R} q_s(y)dy=1$ and $\int_{\R} |\partial_y q_s(y)|dy=cs^{-1/2}$ for a constant $c>0$.}

To show Theorem~\ref{t.ofos}, it is enough to prove that $\{g_{\theta}(0,x;-T)\}_{T\geq0}$ is a Cauchy sequence in $L^p(\Omega)$ \blue{and the  convergence to the limit is exponentially fast}. \blue{For} any $T_2\geq T_1\geq 10$, we decompose the error as follows:
\[
g_{\theta}(0,x;-T_2)-g_{\theta}(0,x;-T_1)=J_0+J_1+J_2-J_3,
\]
with the following definitions of $J_1,J_2,J_3,J_4$:
\begin{align*}
    J_0&=\int_0^1 \phi_{\ic}(y)\bar{\rho}_{\theta}(0,x;T_2,y;-T_2)[\U_\theta(0,x;-T_2,y)-u_\theta(0,x;-T_2)]dy\\
&-\int_0^1 \phi_{\ic}(y)\bar{\rho}_{\theta}(0,x;T_1,y;-T_1)[\U_\theta(0,x;-T_1,y)-u_\theta(0,x;-T_1)]dy,\\
J_1&=\int_0^{\frac12 T_1}\int_0^1 u_{\theta}(-s,y;-T_2)\bar{\rho}_{\theta}(0,x;s,y;-T_2) [\U_\theta(0,x;-s,y)-u_\theta(0,x;-T_2)]dyds\\
&-\int_0^{\frac12 T_1}\int_0^1 u_{\theta}(-s,y;-T_1)\bar{\rho}_{\theta}(0,x;s,y;-T_1) [\U_\theta(0,x;-s,y)-u_\theta(0,x;-T_1)]dyds,\\
J_2&=\int_{\frac12 T_1}^{T_2}\int_0^1 u_{\theta}(-s,y;-T_2)\bar{\rho}_{\theta}(0,x;s,y;-T_2) [\U_\theta(0,x;-s,y)-u_\theta(0,x;-T_2)]dyds,\\
   J_3&=\int_{\frac12 T_1}^{T_1}\int_0^1 u_{\theta}(-s,y;-T_1)\bar{\rho}_{\theta}(0,x;s,y;-T_1) [\U_\theta(0,x;-s,y)-u_\theta(0,x;-T_1)]dyds.
\end{align*}

We have the following two lemmas:
\begin{lemma}\label{l.J23}
    \blue{For any $p\in [1,\infty)$, there exist constants $C,c>0$ depending on $p, R(\cdot)$ such that} 
$\|J_0\|_p+\|J_2\|_p+\|J_3\|_p\leq Ce^{-cT_1}$.
\end{lemma}

\begin{proof}
For $J_0$, we apply Lemma~\ref{l.bdrho}, Proposition~\ref{p.burgersofos} and the fact that $\phi_{\ic}\in L^1(\bT)$ to derive that $\|J_0\|_p\leq Ce^{-cT_1}$. Similarly, applying Lemmas~\ref{l.bdu}, \ref{l.bdrho} and Proposition~\ref{p.burgersofos}, we have 
\[
\|J_2\|_p+\|J_3\|_p \leq C\int_{T_1/2}^\infty e^{-cs}ds,
\]
which completes the proof.
\end{proof}

%
 
\begin{lemma}\label{l.J1}
    \blue{For any $p\in [1,\infty)$, there exist constants $C,c>0$ depending on $p, R(\cdot)$ such that} 
$\|J_1\|_p \leq Ce^{-cT_1}$.
\end{lemma}

\begin{proof}
We can further write $J_1=J_{11}+J_{12}$, with 
\[
\begin{aligned}
J_{11}=\int_0^{\frac12T_1}\int_0^1&[u_{\theta}(-s,y;-T_2)\bar{\rho}_{\theta}(0,x;s,y;-T_2)-u_{\theta}(-s,y;-T_1)\bar{\rho}_{\theta}(0,x;s,y;-T_1)]\\
&\times \U_\theta(0,x;-s,y) dyds=A_1+A_2,
\end{aligned}
\]
where we define
\[
A_1=\int_0^{\frac12 T_1}\int_0^1[u_{\theta}(-s,y;-T_2)-u_{\theta}(-s,y;-T_1)]\bar{\rho}_{\theta}(0,x;s,y;-T_2) \U_\theta(0,x;-s,y) dyds,
\]
and
\[
    \begin{aligned}
        A_2&=\int_0^{\frac12 T_1}\int_0^1 u_\theta(-s,y;-T_1)[\bar{\rho}_{\theta}(0,x;s,y;-T_2)-\bar{\rho}_{\theta}(0,x;s,y;-T_1)]\U_{\theta}(0,x;-s,y)\\&\hspace{4.5in}dyds.
\end{aligned}
\]
For $A_1$, by Proposition~\ref{p.burgersofos} and Lemmas~\ref{l.bdrho} and \ref{l.bdU}, we have 
\[
\begin{aligned}
\|A_1\|_p&\leq C\int_0^{T_1/2}\int_0^1 e^{-c(T_1-s)}(1_{s<1}G_s(y-\theta s-x)+1_{s\geq 1})\big(1+1_{s<1}\tfrac{Q_s(y-\theta s-x)}{G_s(y-\theta s -x)}\big) dyds\\
&\leq Ce^{-cT_1}.
\end{aligned}
\]
For $A_2$, the proof is the same except that we apply Lemma~\ref{l.mixpolymer} rather than Proposition~\ref{p.burgersofos}. This shows that $\|J_{11}\|_p \leq Ce^{-cT_1}$. The proof for $J_{12}$ is very similar, so we omit it.
\end{proof}

Combining Lemmas~\ref{l.J23} and \ref{l.J1}, we conclude that for $T_2\geq T_1\geq 10$, 
\[
\|g_{\theta}(0,x;-T_2)-g_{\theta}(0,x;-T_1)\|_p\leq Ce^{-cT_1}.
\]
 
One can actually write the limit $\tilde{g}_\theta(0,x)=\lim_{T\to\infty} g_\theta(0,x;-T)$ explicitly. Since $\bar{\rho}_\theta(0,x;s,y;-T)$ converges in $L^p(\Omega)$ as $T\to\infty$, in light of  Lemma~\ref{l.mixpolymer}, let us denote the limit by $\bar{\rho}_\theta(0,x;s,y;-\infty)$. \blue{For the first integral on the r.h.s of \eqref{e.exgtheta}, we can apply  Lemma~\ref{l.bdrho} and Proposition~\ref{p.burgersofos} to derive
\[
\int_0^1 \phi_{\ic}(y)\bar{\rho}_{\theta}(0,x;T,y;-T)[\U_\theta(0,x;-T,y)-u_\theta(0,x;-T)]dy\to0
\]
in $L^p(\Omega)$ as $T\to\infty$ (recall that by our assumption $\phi_{\ic}$ is bounded).} Then, as in the above proofs, we can pass to the limit in \eqref{e.exgtheta} and obtain 
\begin{equation}\label{e.explicitg}
\tilde{g}_\theta(0,x)=1+\int_0^\infty\int_0^1\scU_{\theta}(-s,y)\bar{\rho}_{\theta}(0,x;s,y;-\infty) [\U_\theta(0,x;-s,y)-\scU_\theta(0,x)]dyds,
\end{equation}
where we recall that $\scU_{\theta}(t,x)=\lim_{T\to\infty}u_\theta(t,x;-T)$ is the global solution of the stochastic Burgers equation. The proof of Theorem~\ref{t.ofos} is complete.


\section{Proof of Theorems~\ref{t.mainth} and \ref{t.envi}}
\label{s.proofmain}

In this section, we first prove Theorem~\ref{t.mainth}, through an application of Theorem~\ref{t.ofos}. Then we give the proof of Theorem~\ref{t.envi} and discuss the interpretation of $\tilde{g}_\theta(0,0)$ as the Radon-Nikodym derivative of the environment seen from the particle advected by the Burgers flow. We continue to assume that $\xi$ is smooth in space.

We first show that the $\tilde{g}_\theta(t,x)$, obtained in Theorem~\ref{t.ofos}, is a global solution to the Fokker-Planck equation \eqref{e.fk1}. Recall that $g_\theta$ solves \eqref{e.fpg}, so for any $T>0$ and $t>s\geq -T$, by the mild formulation of the Fokker-Planck equation, we have
\begin{equation}\label{e.mildfk}
\begin{aligned}
g_\theta(t,x;-T)=&\int_\bT G_{t-s}(x-y)g_\theta(s,y;-T)dy\\
&+\int_s^t\int_{\bT} \partial_x G_{t-\ell}(x-y)u_{\theta}(\ell,y;-T)g_\theta(\ell,y;-T)dyd\ell.
\end{aligned}
\end{equation}
\blue{Here we have applied integration by parts and the fact that $\partial_x G_{t-\ell}(x-y)=-\partial_y G_{t-\ell}(x-y)$.}
Sending $T\to\infty$ on both sides of \eqref{e.mildfk}, \blue{since $\partial_x G_{t-\ell}(x-y)$ is locally integrable as a function of $(\ell,y)$,} we can apply Theorem~\ref{t.ofos} and Proposition~\ref{p.burgersofos}, we obtain 
\[
\begin{aligned}
\tilde{g}_\theta(t,x)=\int_\bT G_{t-s}(x-y)\tilde{g}_\theta(s,y)dy+\int_s^t\int_{\bT} \partial_x G_{t-\ell}(x-y)\scU_{\theta}(\ell,y)\tilde{g}_\theta(\ell,y)dyd\ell,
\end{aligned}
\]
and this shows that $\tilde{g}_\theta$ is a global solution to \eqref{e.fk1}. 

It is straightforward to check that $\tilde{g}_\theta$ is spacetime stationary. 

Next, we claim that,  if  the initial data is chosen to be $g_\theta(-T,x;-T)=1$, then 
 \begin{equation}\label{e.fkex1}
\EE g_\theta(t,x;-T)=1.
\end{equation}
 This comes from the fact that the Fokker-Planck equation preserves expectation, as can be seen by taking expectations in the mild formulation \eqref{e.mildfk}. More precisely, we have from \eqref{e.mildfk} that
\begin{equation}\label{e.9261}
\begin{aligned}
\EE g_\theta(t,x;-T)=&\int_\bT G_{t-s}(x-y)\EE g_\theta(s,y;-T)dy\\
&+\int_s^t\int_{\bT} G_{t-\ell}(x-y)\EE\big[\partial_y \big(u_{\theta}(\ell,y;-T)g_\theta(\ell,y;-T)\big)\big]dyd\ell.
\end{aligned}
\end{equation}
Since $g_\theta(-T,x;-T)$ is a constant, we know that, for each fixed  $s\geq -T$, $\{g_\theta(s,x;-T)\}_{x\in\bT}$ is stationary, so the expectation $\EE g_\theta(s,x;-T)$ does not depend on the spatial variable, which we will denote by $f(s;-T)$. Furthermore, for each $\ell$, since we know that the process $\{u_{\theta}(\ell,y;-T)g_\theta(\ell,y;-T)\}_{y\in\bT}$ is smooth and stationary, we have \[\EE\big[\partial_y \big(u_{\theta}(\ell,y;-T)g_\theta(\ell,y;-T)\big)\big]=0.\] Combining these two facts, we conclude from \eqref{e.9261} that 
\[
f(t;-T)=f(s;-T), \quad\quad t>s,
\]
which shows that the expectation is preserved and \eqref{e.fkex1} holds. Using \eqref{e.fkex1} and the convergence of $g_\theta(t,x;-T)\to \tilde{g}_\theta(t,x)$ in $L^p(\Omega)$ for any $p\in[1,\infty)$, we conclude that $\EE \tilde{g}_\theta(t,x)=1$. 

Lastly, since $g_\theta(t,x;-T)$ is strictly positive, by the convergence result in Theorem~\ref{t.ofos}, we know that $\tilde{g}_\theta(t,x)\geq0$. Furthermore, since we have $\int_0^1 g_\theta(t,x;-T)dx=1$ for all $T$, it implies that $\int_0^1 \tilde{g}_\theta(t,x)dx=1$ almost surely. Thus, we can conclude that, for any $t\in\R $, the nonnegative function $\tilde{g}_\theta(t,\cdot)$ is not identically zero almost surely. \blue{Applying the strong maximum principle for the Fokker--Planck equation, namely, the strict positivity of the fundamental solution, see \cite[Chapter 2, Theorem 11]{FRI08}, we conclude that $\tilde{g}_\theta$ is strictly positive, almost surely. Alternatively, one can rewrite the Fokker-Planck equation as $\partial_t \tilde{g}_\theta=\frac12\Delta \tilde{g}_\theta+\mathscr{U}_\theta \nabla \tilde{g}_\theta+\tilde{g}_\theta \nabla \mathscr{U}_\theta$, and write the solution by the Feynman-Kac formula as in \cite[7.6 Theorem]{KS14}, and use the strict positivity of the polymer endpoint density to conclude it. }



\subsection{Uniqueness}
In this section, we show that $\tilde{g}_\theta$ is the unique global solution to \eqref{e.fk1}, which is the last piece we need to complete the proof of Theorem~\ref{t.mainth}. This comes from the following proposition: 
\begin{proposition}\label{p.congnew}
If we consider $g_\theta$ as the solution to the following equation:
\begin{equation}\label{e.fk2}
\begin{aligned}
\partial_t 
g_{\theta}(t,x;-T)=\frac12\Delta g_{\theta}(t,x;-T)+\nabla(\scU_\theta(t,x)g_{\theta}(t,x;-T)), \quad\quad t>-T,x\in\bT,
\end{aligned}
\end{equation}
with initial data $g_\theta(-T,x;-T)=g_{\ic}(x)\geq0$ satisfying $\int_{\bT} g_{\ic}(x)dx=1$, then the same result as in Theorem~\ref{t.ofos}   holds, namely, $g_\theta(t,x;-T)$ converges in $L^p(\Omega)$ to  $\tilde{g}_\theta(t,x)$ given by \eqref{e.explicitg}, which does not depend on $g_{\ic}$.
\end{proposition}
Comparing \eqref{e.fk2} with \eqref{e.fk3},  the only difference is the coefficient $u_\theta$ being replaced by $\scU_\theta$. The proof is almost identical, so we only briefly sketch it below and discuss the places where changes are needed.

The key difference is the formula for the quenched mid-point density $\rho_\theta(0,x;s,y;-T)$. \blue{This is because the ``terminal'' condition for the directed polymer is changed from being flat to being stationary.}

By   \cite[Lemma 4.2]{GK23}, we have in this case 
\[
\begin{aligned}
\rho_{\theta}(0,x;s,y;-T)&=\frac{\E[e^{\int_0^T \xi(-t,x+B_t)dt}e^{H_\theta(-T,x+B_T)}\delta_y(x+B_s) ]}{\E[e^{\int_0^T \xi(-t,x+B_t)dt}e^{H_\theta(-T,x+B_T)}]}\\
&=\frac{\cZ_{0,-s}(x,y)\int_{\R} \cZ_{-s,-T}(y,w)e^{H_\theta(-T,w)}dw}{\int_{\R}\cZ_{0,-T}(x,w)e^{H_\theta(-T,w)}dw}.
\end{aligned}
\]
Here $H_\theta(-T,\cdot)$ can be chosen as  any anti-derivative of $\scU_\theta(-T,\cdot)$ (adding any constant to $H_\theta$ does not change the expression). \blue{Note that we chose not to change the notation $\rho_\theta$, although one should distinguish the above expression from the one given in Lemma~\ref{l.exrho}. As mentioned above, the difference is minor and only comes from the ``terminal'' condition. We were hoping that the interested reader would be able to go over the same proof of Theorem~\ref{t.ofos} using the above $\rho_\theta$, without causing any extra confusion.
}

Recall that in Lemma~\ref{l.exrho} of  Section~\ref{s.ofos}, we instead have $H_\theta(-T,x)=\theta x$, which is the anti-derivative of $u_\theta(-T,x;-T)=\theta$. Since $\int_\bT \scU_\theta(t,\cdot) =\theta$, we can assume 
\[
H_\theta(-T,x)=\theta x+\mathscr{H}_\theta(-T,x),
\] where $\mathscr{H}_\theta(-T,\cdot)$ is a $1-$periodic function.

With the above change, \eqref{e.7191} becomes 
\begin{equation}\label{e.exgnew}
\begin{aligned}
g_{\theta}(0,x;-T)&=1+\int_0^1 \phi_{\ic}(y)\bar{\rho}_{\theta}(0,x;T,y;-T)[\U_\theta(0,x;-T,y)-\scU_\theta(0,x)]dy\\
&+\int_0^T\int_0^1\scU_{\theta}(-s,y)\bar{\rho}_{\theta}(0,x;s,y;-T) [\U_\theta(0,x;-s,y)-\scU_\theta(0,x)]dyds,
\end{aligned}
\end{equation}
where there are two changes: (i) $u_\theta\mapsto \scU_\theta$, (ii) the periodized mid-point density in \eqref{e.9271} becomes 
\[
\bar{\rho}_\theta(0,x;s,y;-T)=\frac{\G_{0,-s}^\theta(x,y-\theta s)\int_0^1 \G_{-s,-T}^\theta(y-\theta s,w)e^{\mathscr{H}_\theta(-T,\theta T+w)}dw}{\int_0^1 \G_{0,-T}^\theta(x,w)e^{\mathscr{H}_\theta(-T,\theta T+w)}dw}.
\]
One can check that the same moment estimate in Lemma~\ref{l.bdrho} holds, and moreover, as $T\to\infty$, $\bar{\rho}_\theta(0,x;s,y;-T)$ converges to $\bar{\rho}_\theta(0,x;s,y;-\infty)$, the same limit obtained in Lemma~\ref{l.mixpolymer}. 

Therefore, one can follow the same proof and pass to the limit of $T\to\infty$ in \eqref{e.exgnew} to derive the convergence of $g_\theta(0,x;-T)$, with the limit taking the form 
\[
\tilde{g}_\theta(0,x)=1+\int_0^\infty\int_0^1\scU_{\theta}(-s,y)\bar{\rho}_{\theta}(0,x;s,y;-\infty) [\U_\theta(0,x;-s,y)-\scU_\theta(0,x)]dyds.
\]
This completes the proof of Proposition~\ref{p.congnew}.

\subsection{Environment seen from the particle}

In this section, we provide another interpretation of the random variable $\tilde{g}_\theta(0,0)$, as a Radon-Nikodym derivative associated with a particle advected by the Burgers drift, and complete the proof of Theorem~\ref{t.envi}. Recall that
\[
\frac{d\mathscr{Q}}{d\mathscr{P}}=\tilde{g}_\theta(0,0).
\] and $\omega_t=\tau_{(t,Y_t)}\omega$ with  $Y_t$ solving the SDE 
\[
dY_t=-\scU_\theta(t,Y_t)dt+dB_t, \quad\quad Y_0=0.
\]

\begin{proof}
Throughout the proof, we use $\EE_{\mathscr{Q}}, \EE_{\mathscr{P}}$ for the \blue{total} expectation when $\omega_0$ is sampled from $\mathscr{Q}, \mathscr{P}$ respectively. \blue{In other words, $\EE_{\mathscr{Q}}, \EE_{\mathscr{P}}$ average with respect to all sources of randomnesses, including both the random environments and the driving Brownian motion.}

To show $\mathscr{Q}$ is an invariant measure of $\omega_t$, it suffices to show that for any bounded $F:C(\R\times\bT)\to\R$, it holds that (note $\omega_0=\omega$)
\begin{equation}\label{e.7241}
\EE_{\mathscr{Q}}F(\omega_t)=\EE_{\mathscr{Q}}F(\omega_0), \quad\quad \mbox{ for any } t>0.
\end{equation} 
To show $\omega_t$ converges in distribution to $\mathscr{Q}$, it is enough to show 
\begin{equation}\label{e.7242}
\EE_{\mathscr{P}}F(\omega_t)\to \EE_{\mathscr{Q}}F(\omega).
\end{equation}

First, we can write 
\[
\EE_{\mathscr{Q}}F(\omega_t)=\EE_{\mathscr{Q}}F(\tau_{(t,Y_t)}\omega)=\EE_{\mathscr{P}}F(\tau_{(t,Y_t)}\omega)\tilde{g}_\theta(0,0;\omega),
\]
where we have spelled out  the dependence of $\tilde{g}_\theta$ on $\omega$. We introduce another quantity: let $\mathscr{G}_{t,s}(x,y;\omega)$ be the propagator of the Fokker-Planck equation satisfied by $\tilde{g}_\theta$, namely, for any $t>s$ and $x,y\in\bT$, we have 
\[
\tilde{g}_\theta(t,x;\omega)=\int_{\bT}\mathscr{G}_{t,s}(x,y;\omega)\tilde{g}_\theta(s,y;\omega)dy.
\]
It is clear that one can view $\mathscr{G}$ as the transition density of the underlying diffusion, so the quenched density of $Y_t$ is $\mathscr{G}_{t,0}(\cdot,0;\omega)$. \blue{For the existence of the propagator, or the so-called fundamental solution, in our setting and the interpretation in terms of the quenched density, we refer to \cite[Page 368--369]{KS14} -- one can check the conditions listed in \cite[7.8 Remark]{KS14} apply in our case since the coefficient $\mathscr{U}_\theta$ is smooth and periodic  in the spatial variable. The spatial smoothness of the global solution to the stochastic Burgers equation is a folklore, and one can find a proof e.g.\ in \cite[Section 2.3-2.4]{DGR21}.}

In this way, one can write (recall that $Y_0=0$)
\blue{
\[
\E F(\tau_{(t,Y_t)}\omega) = \int_{\bT} \mathscr{G}_{t,0}(x,0;\omega)F(\tau_{(t,x)}\omega)dx,
\]
where the expectation $\E$ on the left is taken with respect to the driving Brownian motion. In other words, when evaluating the expectation $\EE_{\mathscr{P}}$, we first average out the Brownian motion.} Now we have 
\[
\begin{aligned}
\EE_{\mathscr{Q}}F(\omega_t)&=\EE_{\mathscr{P}}\big[\tilde{g}_\theta(0,0;\omega)\int_{\bT} \mathscr{G}_{t,0}(x,0;\omega)F(\tau_{(t,x)}\omega)dx\big]\\
&=\int_{\bT} \EE_{\mathscr{P}}\big[\tilde{g}_\theta(0,0;\omega)\mathscr{G}_{t,0}(x,0;\omega)F(\tau_{(t,x)}\omega) \big]dx\\
&=\int_{\bT} \EE_{\mathscr{P}}\big[\tilde{g}_\theta(0,0;\tau_{(-t,-x)}\omega)\mathscr{G}_{t,0}(x,0;\tau_{(-t,-x)}\omega)F(\omega) \big]dx.
\end{aligned}
\]
\blue{In the second ``='', we applied Fubini, while in the last ``='' we used the fact that $\tau_{(t,x)}$ preserves $\mathscr{P}$ (the variables $t,x$ are both fixed when we used the stationarity to derive that the two expectations equal to each other).} We note that \[
\begin{aligned}
&\tilde{g}_\theta(0,0;\tau_{(-t,-x)}\omega)=\tilde{g}_\theta(-t,-x;\omega),\\
&\mathscr{G}_{t,0}(x,0;\tau_{(-t,-x)}\omega)=\mathscr{G}_{0,-t}(0,-x;\omega),
\end{aligned}
\]
using which one can further derive
\[
\begin{aligned}
\EE_{\mathscr{Q}}F(\omega_t)&=\int_{\bT} \EE_{\mathscr{P}}\big[\tilde{g}_\theta(-t,-x;\omega)\mathscr{G}_{0,-t}(0,-x; \omega)F(\omega) \big]dx\\
&=\EE_{\mathscr{P}}\big[F(\omega)\int_{\bT}\tilde{g}_\theta(-t,-x;\omega)\mathscr{G}_{0,-t}(0,-x; \omega)dx\big]\\
&=\EE_{\mathscr{P}}\big[F(\omega)\tilde{g}_{\theta}(0,0;\omega)\big]=\EE_{\mathscr{Q}}F(\omega).
\end{aligned}
\]
This completes the proof of \eqref{e.7241}.

To show \eqref{e.7242}, by a similar argument we have
\[
\begin{aligned}
\EE_{\mathscr{P}} F(\omega_t)&=\EE_{\mathscr{P}}\int_{\bT} F(\tau_{(t,x)}\omega)\mathscr{G}_{t,0}(x,0;\omega)dx\\
&=\EE_{\mathscr{P}} F(\omega)\int_{\bT}\mathscr{G}_{t,0}(x,0;\tau_{(-t,-x)}\omega)dx=\EE_{\mathscr{P}} F(\omega)\int_{\bT}\mathscr{G}_{0,-t}(0,-x;\omega)dx.
\end{aligned}
\]
It is easy to see that $\int_{\bT}\mathscr{G}_{0,-t}(0,-x;\omega)dx=g_\theta(0,0;-t)$ with $g_\theta$ solving \eqref{e.fk2} with the initial data $g_\theta(-t,x;-t)\equiv1$. By Proposition~\ref{p.congnew}, we know that $g_\theta(0,0;-t)\to \tilde{g}_\theta(0,0)$ as $t\to\infty$, in $L^p(\Omega)$ for any $p\in[1,\infty)$. This implies that 
\[
\EE_{\mathscr{P}}F(\omega_t)\to \EE_{\mathscr{P}}F(\omega)\tilde{g}_\theta(0,0)=\EE_{\mathscr{Q}}F(\omega),
\]
which completes the proof.
\end{proof}

\section{The   spacetime white noise}
\label{s.white}

In this section, we consider the case when $\xi$ is a spacetime white noise on $\R\times \bT$ and prove Theorem~\ref{t.white}. Now the solution to the stochastic Burgers equation is only distribution-valued, and the Fokker-Planck equation \eqref{e.fk1} is formal since it has distribution-valued coefficients. 

Recall that \[
u_\theta(t,x;-T)=\partial_x \log Z_\theta(t,x;-T),
\]
with $Z_\theta(t,x;-T)$ solving 
the corresponding stochastic heat equation with initial data $Z_\theta(-T,x;-T)=e^{\theta x}$. Using the propagator of SHE, defined in \eqref{e.defZtsxy}, one can write it as 
\begin{equation}\label{e.greenre}
Z_{\theta}(t,x;-T)=\int_\R\cZ_{t,-T}(x,y)e^{\theta y}dy,
\end{equation}
so the differentiability of $Z_\theta$ in $\theta$ is clear. Define the Busemann function
\begin{equation}\label{e.defB}
\begin{aligned}
B_{\theta}(t,x;-T)&:=\int_0^x u_\theta(t,y;-T)dy\\
&=\log Z_\theta(t,x;-T)-\log Z_\theta(t,0;-T).
\end{aligned}
\end{equation}
The convergence of $B_\theta(t,x;-T)$ as $T\to\infty$ is well-known, and the proof can be e.g.\ adapted from the proof of \cite[Theorem 2.3]{GK21} or the one given in Section~\ref{s.pbofos} for Burgers equation. The convergence holds in $L^p(\Omega)$ for any $p\in [1,\infty)$, and we will denote the limit by $\mathscr{B}_\theta(t,x)$. \blue{More precisely, by the same shear transformation as  in Lemma~\ref{l.reduction}, we first reduce it to the case of $\theta=0$. Then by Remark~\ref{r.busemanrho} below, we further reduce it to the study of the polymer endpoint density, and this is precisely the content of \cite[Proposition 4.7]{GK21} and \cite[Theorem 2.3]{GK21}. }

To show Theorem~\ref{t.white}, the idea is the same as before: for any $\lambda$ and $t,x,T$, we write 
\[
B_\lambda(t,x;-T)-B_0(t,x;-T)=\int_0^\lambda \partial_\theta B_\theta(t,x;-T)d\theta,
\]
and try to pass to the limit of $T\to\infty$ on both sides.

For any $\theta\in\R$, by the representation in \eqref{e.greenre}, we have
\begin{equation}
\begin{aligned}
\partial_\theta \log Z_{\theta}(t,x;-T)
=\frac{\int_{\R} \cZ_{t,-T}(x,y)e^{\theta y}ydy}{\int_{\R}\cZ_{t,-T}(x,y)e^{\theta y}dy}.\\
\end{aligned}
\end{equation}
As before, it is enough to consider the case of $t=0$. By an approximation argument (by a mollification of $\xi$, the Feynman-Kac formula and passing to the limit) and the same proof for Lemma~\ref{l.reduction}, we have 
\begin{equation}\label{e.4245}
\partial_\theta\log Z_{\theta}(0,x;-T)=x+\theta T+\psi_\theta(0,x;-T),
\end{equation}
with 
\begin{equation}\label{e.quenchmean}
\psi_{\theta}(0,x;-T)=\frac{\int_{\R} \cZ^\theta_{0,-T}(x,y)(y-x)dy}{\int_{\R} \cZ^\theta_{0,-T}(x,y)dy}.
\end{equation}
We remind the readers that $\cZ^\theta$ is the propagator of the stochastic heat equation driven by $\xi^\theta(t,x)=\xi(t,x-\theta t)$. \blue{For the convenience of the reader, we present a proof of \eqref{e.4245} in the appendix.}

It is straightforward to check that $\psi_\theta$ is $1-$periodic in $x$, and one should view $\psi_\theta(0,x;-T)$ as the quenched mean of the endpoint displacement of the directed polymer, starting at $(0,x)$ and running backward in time. Or equivalently, by periodicity, one may view it as the winding number of the polymer path around a cylinder. 

Therefore, for any $\lambda\in\R, x\in\bT$ and $T>0$, we can write 
\begin{equation}\label{e.7251}
\begin{aligned}
B_\lambda(0,x;-T)-B_0(0,x;-T)&=\int_0^\lambda[\partial_\theta \log Z_{\theta}(0,x;-T)-\partial_\theta\log Z_{\theta}(0,0;-T)] d\theta\\
&=\int_0^\lambda [x+\psi_{\theta}(0,x;-T)-\psi_\theta(0,0;-T)]d\theta\\
&=\lambda x+\int_0^\lambda[\psi_{\theta}(0,x;-T)-\psi_\theta(0,0;-T)]d\theta.
\end{aligned}
\end{equation}
The factor $\lambda x$ simply comes from the drift, so the key is the second term: the quantity $\psi_\theta(0,x;-T)-\psi_\theta(0,0;-T)$ measures the difference between the quenched mean of the polymer paths starting at $(0,x)$ and $(0,0)$ respectively. By the fast mixing of the endpoint distribution of the directed polymer on a torus, one may expect $\psi_\theta(0,x;-T)-\psi_\theta(0,0;-T)$ to stabilize as $T\to\infty$. This is indeed the case and was studied in detail in \cite[Section 3]{GK23}. 

With the above preparations, now we can present the proof:

\begin{proof}[Proof of Theorem~\ref{t.white}]
To show the convergence of $\psi_\theta(0,x;-T)-\psi_\theta(0,0;-T)$ as $T\to\infty$, starting from the expression of \eqref{e.quenchmean} and the fact that $\xi^\theta\stackrel{\text{law}}{=}\xi$, we only need to consider the case of $\theta=0$. This was done in \cite{GK23}. The convergence is a direct consequence of the proof of \cite[Proposition 3.2]{GK23}, see also \cite[Remark 3.10]{GK23}.

To show \eqref{e.lawBB}, it suffices to consider the case of $\theta=0$. To ease the notation, we just write 
\[
\psi(0,x;-T)=\frac{\int_{\R} \cZ_{0,-T}(x,y)(y-x)dy}{\int_{\R} \cZ_{0,-T}(x,y)dy},
\]
and denote the limit of $\psi(0,x;-T)-\psi(0,0;-T)$  by $\tilde{\psi}(0,x)$.
Define a stationary version of $\psi$:
\[
\Psi(0,x;-T)=\frac{\int_{\R} \cZ_{0,-T}(x,y)(y-x)e^{\B(y)}dy}{\int_{\R} \cZ_{0,-T}(x,y)e^{\B(y)}dy},
\]
where $\B$ is a standard Brownian bridge connecting $(0,0)$ and $(1,0)$, periodically extended to $\R$. By \cite[Proposition 3.2]{GK23}, we have 
\[
\|[\psi(0,x;-T)-\psi(0,0;-T)]-[\Psi(0,x;-T)-\Psi(0,0;-T)]\|_p \leq Ce^{-cT}.
\]
This implies that $\Psi(0,x;-T)-\Psi(0,0;-T)\to \tilde{\psi}(0,x)$ in $L^p(\Omega)$, so that it is enough to study the asymptotic distribution of the increments of $\Psi$ (rather than $\psi$). By \cite[Proposition 4.8]{GK23}, we have (for any $T>0$)
\[
\{x+\Psi(0,x;-T)-\Psi(0,0;-T)\}_{x\in\bT}\stackrel{\text{law}}{=}\left\{\int_0^x \rho_T(y)dy\right\}_{x\in\bT},
\]
with the density
\[
\rho_T(y)=\int_{\bT} \frac{\G_{0,-T}(z,y)e^{\B_1(y)}}{\int_{\bT} \G_{0,-T}(z,y')e^{\B_1(y')}dy'} dz,
\]
where $\B_1$ a standard Brownian bridge independent of the noise $\xi$. By an application of \cite[Theorem 2.3]{GK21}, we have $\rho_T(\cdot)\Rightarrow \rho_\infty(\cdot)$ in distribution in $C(\bT)$, with 
\[
\rho_\infty(y)=\frac{  e^{\B_1(y)+\B_2(y)}}{\int_0^1 e^{\B_1(y')+\B_2(y')}dy'}.
\]
This shows that 
\[\{x+\Psi(0,x;-T)-\Psi(0,0;-T)\}_{x\in\bT}\Rightarrow\left\{\int_0^x \rho_\infty(y)dy\right\}_{x\in\bT} 
\] as $T\to\infty$, and completes the proof of \eqref{e.lawBB}.
\end{proof}

 \begin{remark}\label{r.density}
Let us explain the connection between Theorem~\ref{t.white} and Theorem~\ref{t.mainth}, at least on a heuristic level. Starting from the difference between the Busemann functions
\[
\mathscr{B}_\lambda(t,x )-\mathscr{B}_0(t,x )=\int_0^\lambda [x+\tilde{\psi}_\theta(t,x)]d\theta, 
\]
one can take the derivative with respect to $x$ on both sides to derive
\[
\mathscr{U}_\lambda(t,x)-\mathscr{U}_0(t,x)=\int_0^\lambda[1+\partial_x\tilde{\psi}_\theta(t,x)]d\theta.
\]
This corresponds to \eqref{e.mainidentity} in the statement of Theorem~\ref{t.mainth}, with $\tilde{g}_\theta(t,x)$ replaced by $1+\partial_x \tilde{\psi}_\theta(t,x)$. Further differentiating in $x$ on \eqref{e.lawBB} leads to 
\begin{equation}\label{e.lawBB1}
\{1+\partial_x\tilde{\psi}_\theta(t,x)\}_{x\in\bT}\stackrel{\text{law}}{=}\left\{\frac{ e^{\B_1(x)+\B_2(x)}}{\int_0^1 e^{\B_1(y)+\B_2(y)}dy}\right\}_{x\in\bT}.
\end{equation}
Thus, in the white noise case, the distribution of $\tilde{g}_\theta(t,\cdot)$ is explicit (for fixed $t$ and $\theta$), given by the r.h.s.\ of \eqref{e.lawBB1}, which is actually the ``stationary distribution'' of a directed polymer in the white noise environment. To see why it holds, let us recall that $\tilde{g}_\theta(t,\cdot)$ is the ``stationary distribution'' of the diffusion in the Burgers drift 
\[
dY_t=-\scU_\theta(t,Y_t)dt+dB_t.
\] In the white noise setting, $\{-\scU_\theta(t,x)\}_{t,x}$ has the same law as $\{\scU_\theta(-t,x)\}_{t,x}$, through which we can relate $Y_t$ to another diffusion 
\[
d\tilde{Y}_t=\scU_\theta(-t,Y_t)dt+dB_t,
\] but this is just another way of describing a directed polymer in the white noise environment, see \cite[Lemma 4.2]{GK23}. For $\tilde{Y}_t$, through the Gibbs measure formulation, one can derive the density expression on the r.h.s.\ of \eqref{e.lawBB1}.
\end{remark}

\appendix


\section{ Burgers equation with periodic forcing}
\label{s.aburgers}

In this section, we present some standard results on the stochastic Burgers equation with a \blue{smooth and  periodic forcing}. The one-force-one-solution principle stated below is based on \cite{GK21}, which dealt with a closely related object, the endpoint distribution of the directed polymer on a cylinder, and the proof is inspired by Sinai's classical work \cite{sinai}.

\blue{In this section, we consider the case of $\xi$ being smooth in the spatial variable.}

\subsection{Moment estimates}
\blue{Recall that $u_\theta$ solves the stochastic Burgers equation \eqref{e.burgers}, with the initial data $u_{\theta}(-T,x;-T)=\theta$. The following lemma provides a moment bound on $u_\theta$, that is uniform in $t$.}
\begin{lemma}\label{l.mmbd}
    For any $p\in [ 1,\infty)$, there exists $C=C(p,R)>0$ such that 
\begin{equation}\label{e.bdu0422}
\sup_{t\geq0 }\EE |u_0(t,0;0)|^p \leq C.
\end{equation}
\blue{In addition, we have 
\begin{equation}\label{e.aid}
\sup_{t>1,x,y\in\bT} \EE |\partial_x \log \G_{t,0}(x,y)|^p\leq C.
\end{equation}}
\end{lemma}

\begin{proof}
For $t\in[0,1]$, we write the solution by the Feynman-Kac formula:
\[
u_0(t,0;0)=\frac{\E e^{\int_0^t \xi(t-s,B_s)ds}\int_0^t\partial_x\xi(t-s,B_s)ds}{\E e^{\int_0^t \xi(t-s,B_s)ds}}.
\]
 For the denominator we have the negative moment bound 
\[
\sup_{t\in[0,1]}\EE |\E e^{\int_0^t \xi(t-s,B_s)ds}|^{-n} \leq C
\]
\blue{for arbitrary $n\in\Z_+$, which can be shown  using the Jensen's inequality and the calculation of the Gaussian moment generating function: \blue{ we have 
\[
\EE |\E e^{\int_0^t \xi(t-s,B_s)ds}|^{-n} \leq \EE\E e^{-n\int_0^t \xi(t-s,B_s)ds}=e^{\frac12n^2R(0)t}.
\] The last step comes from the fact that for each fixed realization of $B$, the random variable $\int_0^t \xi(t-s,B_s)ds\sim N(0,R(0)t)$. }
For the numerator, we apply H\"older inequality to derive 
\[
\begin{aligned}
&\EE |\E e^{\int_0^t \xi(t-s,B_s)ds}\textstyle\int_0^t\partial_x\xi(t-s,B_s)ds|^{2n}\\
&\leq \EE \E e^{2n\int_0^t \xi(t-s,B_s)ds}(\textstyle\int_0^t\partial_x\xi(t-s,B_s)ds)^{2n}\\
& \leq \sqrt{\EE \E e^{4n\int_0^t \xi(t-s,B_s)ds}}\sqrt{\EE\E (\textstyle\int_0^t\partial_x\xi(t-s,B_s)ds)^{4n}}.
\end{aligned}
\]
The first factor is computed in the same way as before. For the second factor,  for each fixed realization of the Brownian motion, we have $\int_0^t \partial_x \xi(t-s,B_s)\sim N(0, -R''(0)t)$, so it is enough to evaluate the expectation $\EE$ first to complete the proof for $t\in[0,1]$.}

For $t>1$, we employ a different proof. The Hopf-Cole transformation says that 
\[
u_0(t,x;0)= \partial_x \log Z_0(t,x;0),
\]
where $Z_0$ solves the stochastic heat equation with initial data $Z_0(0,x;0)=1$, so we can view $Z_0$ as the solution to the equation on the torus. By the propagator $\G$, we can write 
\[
Z_0(t,x;0)=\int_{\bT} \G_{t,t-1}(x,y)Z_0(t-1,y;0)dy,
\]
and 
\[
\partial_x Z_0(t,x;0)=\int_{\bT}\partial_x \G_{t,t-1}(x,y)Z_0(t-1,y;0)dy.
\]
This implies that 
\[
|\partial_x \log Z_0(t,x;0)| \leq \frac{\sup_y |\partial_x \G_{t,t-1}(x,y)|}{ \inf_y \G_{t,t-1}(x,y)}.
\]
By \cite[Lemma 4.1]{GK21}, we have $\EE |\inf_y \G_{t,t-1}(x,y)|^{-p} \leq C$. Applying Lemma~\ref{l.bddegreen} below, we have $\EE \big(\sup_y |\partial_x \G_{t,t-1}(x,y)|\big)^p \leq C$. This completes the proof of \eqref{e.bdu0422}. The same proof also implies \eqref{e.aid}.
\end{proof}


The following lemma is on the moment bounds on the derivatives of the stochastic heat equation propagator.

\begin{lemma}\label{l.bddegreen}
For $t\in(0,2]$, we have 
\begin{equation}\label{e.7201}
\begin{aligned}
&\|\G_{t,0}(x,y)^{-1}\|_p \leq C G_t(x-y)^{-1},\\
&\|\partial_x \G_{t,0}(x,y)\|_p \leq C\big(G_t(x-y)+Q_t(x-y)\big).
\end{aligned}
\end{equation}
For $t\in[\frac12,2]$, we have 
\begin{equation}\label{e.7202}
\|\partial_x\partial_ y\G_{t,0}(x,y)\|_p  \leq C.
\end{equation}
\end{lemma}

\begin{proof}
First, we can write $\G_{t,0}(x,y)$ as 
\[
\begin{aligned}
&\G_{t,0}(x,y)=\sum_n \cZ_{t,0}(x,y+n)=\sum_n\E \big[e^{\int_0^t \xi(t-s,x+B_s)ds-\frac12R(0)t}\delta_{y+n}(x+B_t)\big]\\
&=\sum_nq_t(x-y-n)\E \big[e^{\int_0^t \xi(t-s,x+B_s)ds-\frac12R(0)t}\,|\,x+B_t=y+n\big]=\sum_n q_t(x-y-n)X_n,
\end{aligned}
\]
where we have \blue{denoted} the term $\E[\cdot]$ by $X_n$ to simplify the notation. 

To estimate $\EE \G_{t,0}(x,y)^{-p}$, we apply Jensen's inequality to derive 
\[
\EE \G_{t,0}(x,y)^{-p} \leq G_t(x-y)^{-p}\sum_n \frac{q_t(x-y+n)}{G_t(x-y)}\EE X_n^{-p}.
\]
By applying Jensen's inequality again and a Gaussian calculation, it is straightforward to check that $\EE X_n^{-p}\leq C$, and this completes the proof of the first inequality in \eqref{e.7201}. 

To estimate the derivative in $x$, we first note that, under the condition $x+B_t=y+n$, the process $\{x+B_s\}_{s\in[0,t]}$ is a Brownian bridge connecting $x$ and $y+n$, so one can write it as $x+B_s=x+\tfrac{y+n-x}{t}s+\tilde{B}_s$ where $(\tilde{B}_s)_{s\in[0,t]}$ is a Brownian bridge connecting $0$ and $0$. This leads to the expression
\[
\G_{t,0}(x,y)=\sum_n q_t(x-y-n) \E \big[e^{\int_0^t \xi(t-s,x+\tfrac{y+n-x}{t}s+\tilde{B}_s)ds-\frac12R(0)t}\,|\,\tilde{B}_t=0\big].
\]
Taking the derivative in $x$ gives 
\begin{equation}\label{e.7203}
\begin{aligned}
&\partial_x \G_{t,0}(x,y)=\sum_n \partial_xq_t(x-y-n)\E \big[e^{\int_0^t \xi(t-s,x+\tfrac{y+n-x}{t}s+\tilde{B}_s)ds-\frac12R(0)t}\,|\,\tilde{B}_t=0\big]\\
&\qquad\qquad\quad+\sum_nq_t(x-y-n)\E \Big[e^{\int_0^t \xi(t-s,x+\tfrac{y+n-x}{t}s+\tilde{B}_s)ds-\frac12R(0)t}\\&\qquad\qquad\qquad\qquad\qquad\qquad\qquad\cdot\textstyle\int_0^t \partial_x \xi(t-s,x+\tfrac{y+n-x}{t}s+\tilde{B}_s)(1-\tfrac{s}{t})ds\,\Big|\,\tilde{B}_t=0\Big],
\end{aligned}
\end{equation}
from which we obtain the moment estimates in \eqref{e.7201} (for $t$ in a bounded interval)
\[
\begin{aligned}
\|\partial_x \G_{t,0}(x,y)\|_p&\leq C\big(G_t(x-y)+\sum_n |\partial_xq_t(x-y-n)|\big)\\
&=C \big(G_t(x-y)+Q_t(x-y)\big).
\end{aligned}
\]

To show \eqref{e.7202}, it remains to take the derivative in $y$ in \eqref{e.7203}. Since the rest of the proof is rather standard, we do not provide it here. 
\end{proof}

\subsection{One-force-one-solution principle: proof of Proposition~\ref{p.burgersofos}}
\label{s.pbofos}
The goal of this section is to prove Proposition~\ref{p.burgersofos}. 

 For any $s<t$, let $\F_{[s,t]}$ be the $\sigma-$algebra generated by the noise $\{\xi(\ell,x)\}_{\ell\in[s,t],x\in\bT}$. Take any measure $\nu$ on $\bT$, consider the solution to the Burgers equation given by the Hopf-Cole transformation 
\[
u(t,x)=\partial_x \log \int_{\bT} \G_{t,0}(x,y)\nu(dy).
\]
We will show
\begin{proposition}\label{p.burgersapp}
For $t\geq 100$ and any $k\leq t-10$, there exists $u_k(t,\cdot)\in C(\bT)$ which is $\F_{[t-k,t]}-$measurable and does not depend on $\nu$, such that for any $p\in[1,\infty)$, 
\begin{equation}\label{e.expmixing}
\sup_{\nu} \EE \|u(t,\cdot)-u_k(t,\cdot)\|_{C(\bT)}^p \leq Ce^{-ck}.
\end{equation}
\end{proposition}
Combined with  Lemma~\ref{l.reduction}, the above proposition implies both \eqref{e.7211} and \eqref{e.7212}.

The proof of the above proposition follows closely the proof of \cite[Proposition 4.7]{GK21}, almost verbatim. In the following, we will only sketch it and refer to the detailed arguments in \cite{GK21} when necessary.

Let $Z(t,x)=\int_{\bT} \G_{t,0}(x,y)\nu(dy)$ be the solution to the stochastic heat equation started at $\nu$. Recall that $\G$ is the propagator of the equation on torus. By \cite[Eq. (4.16)]{GK21}, we know that 
\[
Z(t,x)=c_{t,\nu}(\omega) \E^\omega[\G_{t,t-1}(x,Y_1)],
\]
where $c_{t,\nu}(\omega)$ is a random constant that depends on $t,\nu$, $\{Y_k\}_{k\geq 1}$ is a Markov chain running backwards in time, whose transition kernel is constructed  through $\G$ (hence depends on the noise $\xi$), see \cite[Section 4]{GK21}, and $\E^\omega$ is the expectation on the Markov chain. Here we have used $\omega$ to denote the random environment $\xi(\cdot,\cdot)$.

In this way, one can write the solution to the Burgers equation as  
\begin{equation}\label{e.9272}
u(t,x)= \frac{\partial_x Z(t,x)}{Z(t,x)}=\frac{\E^\omega[\partial_x\G_{t,t-1}(x,Y_1)]}{\E^\omega[\G_{t,t-1}(x,Y_1)]},
\end{equation}
so that the random constant $c_{t,\nu}(\omega)$ cancels.

In \cite[Section 4.5]{GK21}, an approximation of $X(t,\cdot):=\E^\omega[\G_{t,t-1}(\cdot,Y_1)]$ was constructed, denoted by $X_k(t,\cdot)$, such that $X_k(t,\cdot)$ is $\F_{[t-k,t]}-$measurable, independent of $\nu$, and 
\begin{equation}\label{e.7213}
\EE \sup_{\nu} \|X(t,\cdot)-X_k(t,\cdot)\|_{C(\bT)}^p \leq Ce^{-ck},
\end{equation}
see \cite[Lemma 4.6]{GK21}. The same proof actually shows that 
\begin{equation}\label{e.7214}
\EE \sup_{\nu} \|\partial_x X(t,\cdot)-\partial_x X_k(t,\cdot)\|_{C(\bT)}^p \leq Ce^{-ck}.
\end{equation}

Now one can construct the approximation of $u(t,\cdot)$ given in \eqref{e.9272}. It was known from \cite[Eq. (4.53)]{GK21} that there exists $\delta\in(0,1)$ and a set $C_k\in \F_{[t-k,t]}$ such that  
\[
\inf_{x\in\bT} X_k(t,x)\geq \delta \inf_{x,y\in\bT} \G_{t,t-1}(x,y), \quad\quad \omega\in C_k, 
\]
and $\PP[C_k^c] \leq (1-\delta)^k$. Define 
\[
u_k(t,x)=\frac{\partial_x X_k(t,x)}{X_k(t,x)}1_{C_k}(\omega).
\]
Following the same proof as \cite[Proposition 4.7]{GK21}, we derive \eqref{e.expmixing}.

\blue{
\begin{remark}\label{r.busemanrho}
In the case of a spacetime white noise, the solution to the stochastic Burgers equation is distribution-valued, and the expression \eqref{e.expmixing} makes no sense. In this case, one may instead consider the increments of the KPZ solution 
\[
B(t,x):=\log \int_{\bT}\G_{t,0}(x,y)\nu(dy)- \log \int_{\bT} \G_{t,0}(0,y)\nu(dy),\quad\quad t>0, x\in\bT.
\]
It can be rewritten using the forward polymer endpoint density as follows
\[
\begin{aligned}
B(t,x)&= \log\frac{\int_{\bT}\G_{t,0}(x,y)\nu(dy)}{\int_{\bT^2} \G_{t,0}(x',y)\nu(dy)dx'}- \log \frac{\int_{\bT} \G_{t,0}(0,y)\nu(dy)}{\int_{\bT^2} \G_{t,0}(x',y)\nu(dy)dx'}\\
&=\log \rhof(t,x;0,\nu)-\log \rhof(t,0;0,\nu).
\end{aligned}
\]
Thus, the same result as Proposition~\ref{p.burgersapp} holds for $B(t,x)$, once we apply \cite[Proposition 4.7]{GK21}.
\end{remark}
}

\section{\blue{Technical lemmas}}

\blue{First, we have the following lemma on the (pathwise) smoothness of the solution to the stochastic heat equation. By the Hopf-Cole transformation, the same holds for  the solution to the stochastic Burgers equation.}

\blue{\begin{lemma}\label{l.jointsmooth}
Fix $t>0$. With probability $1$, the function $Z_\theta(t,x;0)$ is smooth in $(\theta,x)$. 
\end{lemma}}

\blue{\begin{proof}
Recall that $Z_\theta(t,x;0)$ admits the Feynman-Kac representation
\[
Z_\theta(t,x;0)=\E \left[e^{\int_0^t\xi(t-s,x+B_s)ds-\frac12R(0)t}e^{\theta (x+B_t)}\right].
\]
To show $Z_\theta(t,x;0)$ is smooth in $(\theta,x)$, it suffices to justifying the interchange of taking the expectation and taking the derivative, since, for each realization of $B$, the function 
\[
f(x,\theta,B):=e^{\int_0^t\xi(t-s,x+B_s)ds-\frac12R(0)t}e^{\theta (x+B_t)}
\] is smooth in $(\theta,x)$ (here one can fix the realization of $\xi$). In the following, we  discuss the differentiability in $x$ as an example -- the one in $\theta$ and higher derivatives can be analyzed in the same way. We claim that for any $a<b$,
\begin{equation}\label{e.4233}
\E \sup_{x\in [a,b]}\partial_x f(x,\theta,B)<\infty.
\end{equation}
With \eqref{e.4233}, one can apply the dominated convergence theorem to move $\partial_x$ inside the expectation and derive 
\[
\partial_x Z_\theta(t,x;0)=\E \partial_x f(x,\theta,B).
\] To show \eqref{e.4233}, we write it as 
\[
\partial_x f(x,\theta,B)=\partial_x f(a,\theta,B)+\int_a^x \partial_x^2 f(x',\theta,B)dx',
\]
so it remains to prove
\begin{equation}\label{e.4234}
\E \left[|\partial_x f(a,\theta,B)|+\int_a^b |\partial_x^2 f(x',\theta,B)|dx'\right]<\infty.
\end{equation}
Taking expectation with respect to $\xi$, by a standard calculation involving Gaussian random variables, we have 
\[
\EE\E \left[|\partial_x f(a,\theta,B)|+\int_a^b |\partial_x^2 f(x',\theta,B)|dx'\right]<\infty,
\]
which implies that \eqref{e.4234} and \eqref{e.4233} hold almost surely. The proof is complete.
\end{proof}}

\blue{Next, we summarize some estimates that were used in the paper concerning the endpoint  density of the directed polymer on the cylinder.}

\blue{\begin{lemma}\label{l.estipolymer}
For any $p\in[1,\infty)$, there exists $C(p,R)$ such that  
\begin{equation}\label{e.4241}
\|\sup_{x\in\bT}\rhof(t,x;0,\nu)\|_p+\|\sup_{x\in\bT} \rhof(t,x;0,\nu)^{-1}\|_p \leq C, \quad\quad t>1, \nu\in \mathcal{M}_1(\bT),
\end{equation}
and
\begin{equation}\label{e.4242}
\|\rhof(t,x;0,y)\|_p \leq C G_t(x-y), \quad \|\rhof(t,x;0,y)^{-1}\|_p \leq C G_t(x-y)^{-1}, \quad\quad t\in(0,1].
\end{equation}
If $\nu=\mathrm{m}$, then \eqref{e.4241} holds for all $t>0$. The same results hold for $\rhob$.
\end{lemma}}

\begin{proof}
\blue{First, for $t>1$, we write the forward endpoint density as 
\[
\rhof(t,x;0,\nu)=\frac{\int_{\bT} \G_{t,t-\frac12}(x,x')\rhof(t-\frac12,x';0,\nu)dx'}{\int_{\bT^2} \G_{t,t-\frac12}(x'',x')\rhof(t-\frac12,x';0,\nu)dx'dx''},
\]
which implies that 
\[
\begin{aligned}
\inf_{x,x'\in\bT} \G_{t,t-\frac12}(x,x')&\inf_{x',x''\in\bT}\G_{t,t-\frac12}^{-1}(x'',x')\\
&\leq  \rhof(t,x;0,\nu) \leq \sup_{x,x'\in\bT} \G_{t,t-\frac12}(x,x') \sup_{x',x''\in\bT}\G_{t,t-\frac12}^{-1}(x'',x').
\end{aligned}
\]
Applying \cite[Lemma 4.1]{GK21}, we derive \eqref{e.4241}.}

\blue{To show \eqref{e.4242}, we know  from  \eqref{e.7201} that $\|\G_{t,0}(x,y)^{-1}\|_p \leq C G_t(x-y)^{-1}$. The same proof gives $\|\G_{t,0}(x,y)\|_p \leq C G_t(x-y)$. 
%
Recall that $\rhof$ is given by
\begin{equation}\label{e.4243}
\rhof(t,x;0,y)=\frac{\G_{t,0}(x,y)}{\int_{\bT} \G_{t,0}(x',y)dx'}.
\end{equation}
For the denominator, a similar proof as in Lemma~\ref{l.bddegreen} (alternatively, applying \cite[Lemma B.2]{GK21})  leads to 
\[
\|\int_{\bT} \G_{t,0}(x',y)dx'\|_p+\|(\int_{\bT} \G_{t,0}(x',y)dx')^{-1}\|_p \leq C, \quad\quad t\in(0,1].
\]
Thus, one can apply H\"older inequality to the r.h.s.\ of \eqref{e.4243}  to conclude \eqref{e.4242}. To prove \eqref{e.4241} in the case of $\nu=\mathrm{m}$ and $t\in(0,1]$, we  apply the above estimates again. The proof is complete.}
\end{proof}

\blue{In the end, we present a proof of \eqref{e.4245}, which relates $\partial_\theta \log Z_\theta$ to the quenched mean of the polymer endpoint in the white noise setting.}

\blue{\begin{lemma}
In the white noise case, we have 
\[
\partial_\theta\log Z_{\theta}(0,x;-T)=x+\theta T+\psi_\theta(0,x;-T),
\]
with 
\[
\psi_{\theta}(0,x;-T)=\frac{\int_{\R} \cZ^\theta_{0,-T}(x,y)(y-x)dy}{\int_{\R} \cZ^\theta_{0,-T}(x,y)dy}.
\]
\end{lemma}
\begin{proof}
Recall that \[
\partial_\theta \log Z_{\theta}(0,x;-T)
=\frac{\int_{\R} \cZ_{0,-T}(x,y)e^{\theta y}ydy}{\int_{\R}\cZ_{0,-T}(x,y)e^{\theta y}dy}.
\]
For any $\eps>0$, let $\xi_\eps$ be a spatial mollification of (the white noise) $\xi$ on the scale of $\eps$, and $\cZ_{t,s,\eps}(x,y)$ be the propagator of the SHE with noise $\xi_\eps$. Consider the r.h.s.\ of the above display with $\cZ$ replaced by $\cZ_\eps$:
\[
\frac{\int_{\R} \cZ_{0,-T,\eps}(x,y)e^{\theta y}ydy}{\int_{\R}\cZ_{0,-T,\eps}(x,y)e^{\theta y}dy}.
\]
For the numerator, through the Feynman-Kac formula, we have 
\[
\int_{\R} \cZ_{0,-T,\eps}(x,y)e^{\theta y}ydy=\E\left[ e^{\int_0^T \xi_\eps(-s,x+B_s)ds-\frac12R_\eps(0)T}e^{\theta (x+B_T)}(x+B_T)\right],
\]
where $R_\eps$ is the spatial covariance function of $\xi_\eps$. By the same argument as in Lemma~\ref{l.reduction} (applying Cameron-Martin theorem), the r.h.s.\ is equal to 
\[
\begin{aligned}
e^{\frac12\theta^2 T}\E\left[ e^{\int_0^T \xi_\eps^\theta(-s,x+B_s)ds-\frac12R_\eps(0)T}e^{\theta x}(x+\theta T+B_T)\right].
\end{aligned}
\]
The same discussion applies to the denominator and in the end we obtain 
\[
\begin{aligned}
\frac{\int_{\R} \cZ_{0,-T,\eps}(x,y)e^{\theta y}ydy}{\int_{\R}\cZ_{0,-T,\eps}(x,y)e^{\theta y}dy}&=x+\theta T+\frac{\E\left[ e^{\int_0^T \xi_\eps^\theta(-s,x+B_s)ds-\frac12R_\eps(0)T} B_T\right]}{\E\left[ e^{\int_0^T \xi_\eps^\theta(-s,x+B_s)ds-\frac12R_\eps(0)T}\right]}\\
&=x+\theta T+\frac{\int_{\R} \cZ_{0,-T,\eps}^\theta (x,y)(y-x)dy}{\int_{\R}\cZ_{0,-T,\eps}^\theta(x,y) dy},
\end{aligned}
\]
where $\cZ_\eps^\theta$ is the propagator of the SHE driven by $\xi_\eps^\theta$.
Sending $\eps\to0$, the proof is complete.
\end{proof}}


 \end{document}